\documentclass[11pt]{amsart} 
\usepackage{enumerate}
\usepackage{amsmath}
 \ifx\pdftexversion\undefined
 \usepackage[dvips]{graphicx,color}
 \else
 \usepackage[pdftex]{graphicx,color}
 \fi

\newcommand{\andname}{and}
\newcommand{\volumename}{volume}
\newcommand{\Inname}{in}
\newcommand{\ofname}{of}
\newcommand{\pagesname}{pages}

\let\ifanglais\iftrue

\def\R{{\mathbb R}}
\def\N{{\mathbb N}}

\def\vol{{\rm Vol}}

\newcommand{\vect}[1]{\overrightarrow{#1}} 

\newcommand{\ed}[1]{\textrm{d} #1}

\DeclareMathOperator{\arctanh}{arctanh}
\newcommand{\asb}{AsB}
\newcommand{\cA}{\mathcal A}

\newtheoremstyle{mesthm}
  {10pt plus 1pt minus 1pt}
  {9pt plus 1pt minus 6pt}
  {\slshape}
  {0.5cm}
  {\bfseries}
  {.}
  {1ex}
  {}
\newtheoremstyle{mesdefi}
  {6pt plus 1pt minus 1pt}
  {6pt plus 1pt minus 1pt}
  {}
  {0.5cm}
  {\bfseries}
  {.}
  {1ex}
  {}%

\theoremstyle{mesthm}
\newtheorem{lema}{\ifanglais{\large L}emma\else{\large L}emme\fi}
\newtheorem{theo}[lema]{\ifanglais{\large T}heorem\else {\large
    T}h\'eor\`eme\fi}
\newtheorem*{theo*}{\ifanglais{\large T}heorem\else {\large T}h\'eor\`eme\fi}

\newtheorem{prop}[lema]{{\large P}roposition} 
 
\newtheorem{cor}[lema]{{\large C}orollary}
\newtheorem*{cor*}{{\large C}orollary}

\newtheorem{claim}[lema]{\ifanglais Claim\else Affirmation\fi}

\theoremstyle{mesdefi}
\newtheorem{defi}[lema]{\ifanglais{\large D}efinition\else{\large
    D}\'efinition\fi} 
\newtheorem*{defi*}{\ifanglais Definition\else D\'efinition\fi}

\newtheorem{rmq}[lema]{\ifanglais{\large R}emark\else{\large
    R}emarque\fi}

\newcounter{step}

\DeclareMathOperator{\Ent}{\overline{Ent}}
\DeclareMathOperator{\ent}{Ent}
\DeclareMathOperator{\lent}{\underline{Ent}}
\DeclareMathOperator{\Area}{Area}

\title[Approximability and volume entropy]%
{Approximability of convex bodies and volume entropy in Hilbert geometry}
\author[C. Vernicos]{Constantin Vernicos}
\email{Constantin.Vernicos@umontpellier.fr}
\address{%
  Institut montpelierain Alexander Grothendieck\\ 
  Universit\'e de Montpellier\\
  Case Courrier 051\\
  Place Eug\`ene Bataillon \\
  F--34395 Montpellier Cedex\\ 
  France} 


\thanks{The author acknowledges that this material is based upon works partially supported by the Science Foundation Ireland under a Stokes award}

\begin{document}

\begin{abstract}
The approximability of a convex body is a number which measures the difficulty in approximating
that convex body by polytopes. In the interior of a convex body one can define its Hilbert
geometry. 
We prove on the one hand that the volume entropy is twice the approximability for a  Hilbert geometry in dimension two end three, and on the other hand that in higher dimensions it is a lower bound of the entropy.
 As a corollary we solve the volume entropy upper bound conjecture in dimension
three and give a new proof in dimension two from the one given in \cite{berck_Bernig_Vernicos}. Moreover, our method allows us to prove the existence of Hilbert geometries with intermediate volume growth one the one hand,
and that in general the volume entropy is not a limit on the other hand.
\end{abstract}

\maketitle

\section*{Introduction and statement of results}

Hilbert geometries are all the metric spaces obtained by defining the 
so-called Hilbert distance on open bounded convex sets in $\R^n$. 
The definition of this distance
uses cross-ratios in the same way as in Klein projective model of 
the hyperbolic geometry \cite{dhilbert}. These metric spaces are actually 
length space whose structure is defined by a Finsler metric which is Riemannian
if and only if the underlying open bounded convex set is an ellipsoid \cite{dckay}. 

These geometries have attracted a lot of interest
see for example the works of
 Y.~Nasu~\cite{nasu61},  W.~Goldmann~\cite{Goldman-1990}, 
P.~de~la~Harpe~\cite{dlharpe},
A.~Karlsson and G.~Noskov~\cite{Karlsson-Noskov-2002}, 
Y.~Benoist~\cite{be,yvessurvey}, 
T.~Foertsch and A.~Karlsson~\cite{foertsch-karlsson05}, 
I.~Kim~\cite{Kim-2005},
B.~Colbois, C.~Vernicos and P.~Verovic~\cite{cvv1,cvv2}, 
B.~Lins and R.~Nussbaum~\cite{Lins-Nussbaum-2008},
A.~Borisenko and E.~Olin~\cite{Borisenko-Olin-2008,Borisenko-Olin-2011}, 
B.~Lemmens and C.~Walsh~\cite{Lemmens-Walsh-2011},
C.~Vernicos~\cite{ver09,ver2011,ver2012},
L.~Marquis~\cite{ Marquis-2012},
M.~Crampon and L.~Marquis~\cite{Crampon-Marquis-2014},  D.~Cooper, D.~Long and S.~Tillman~\cite{Cooper-Long-Tillmann}), X.~Nie~\cite{Nie-2011} and the Handbook of Hilbert geometry~\cite{Handbook}.

The present paper focuses on the volume growth of these geometries and more specifically on the volume entropy.

Let $\Omega$ be a bounded open convex set in $\R$ endowed with its Hilbert geometry. If we consider the \textsl{Busemann} volume  $\vol_\Omega$ and 
denote by $B_\Omega(p,r)$ the metric ball of radius $r$ centred at the
point $p\in \Omega$, then the \textsl{lower and upper volume entropies} of 
$\Omega$ will be defined respectively by
\begin{multline}
     \lent(\Omega)=\liminf_{r\to +\infty} \dfrac{\ln\bigl(\vol_\Omega(B_\Omega(p,r)\bigr)}{r}\text{, and }\\  \Ent(\Omega)=\limsup_{r\to +\infty} \dfrac{\ln\bigl(\vol_\Omega(B_\Omega(p,r)\bigr)}{r}\text{.}
  \end{multline}
When the two limits coincide we denote their common limit by 
$\ent(\Omega)$ and call it the \textsl{volume entropy} of $\Omega$.

Let us stress out that in this definition the upper and lower volume entropy of 
$\Omega$ do not depend
on the base point $p$ and are
actually projective invariant attached to $\Omega$.

The question we address in this essay is twofold. On
the one hand it is an investigation of the existence of
an analogue, for all Hilbert geometries, of
the relation  between the volume entropy and the Hausdorff dimension
of the radial limit set on the universal cover of a compact Riemannian manifold 
with non-positive curvature.
On the other hand we focus on 
the \textsl{volume entropy upper bound conjecture}
which states that if $\Omega$ is an open and bounded convex subset of $\R^n$, then $\Ent(\Omega)\leq n-1$.
To put our work into perspective let us recall the main related results.

The first one is a complete answer to  the conjecture in the two-dimensional case
by G.~Berck, A.~Bernig and C.~Vernicos in~\cite{berck_Bernig_Vernicos}, 
where the authors actually
obtained an upper bound as a function of $d$, the upper Minkowski dimension (or \textsl{ball-box} dimension) of the set of extreme points of $\Omega$, namely 
\begin{equation}
 \Ent(\Omega) \leq \frac{2}{3-d} \leq 1.
\end{equation}

The second result is a more precise statement with respect to the asymptotic volume growth of balls.
It involves another projective invariant introduced by G.~Berck, A.~Bernig and C.~Vernicos in the introduction of~\cite{berck_Bernig_Vernicos} called the  \textsl{centro-projective}  area of $\Omega$, and defined by
\begin{equation}
 \cA_p(\Omega):=\int_{\partial \Omega}\frac{\sqrt{k(x)}}{\langle n(x),x-p\rangle^{\frac{n-1}{2}}}\left(\frac{2(\alpha x)}{1+a(x)}\right)^{\frac{n-1}{2}}\ dA(x),
\end{equation}
where for any $x\in \partial\Omega$,  $k(x)$ is the Gauss curvature, $n(x)$ the outward normal and $\alpha(x)>0$ is the function  defined by $p-\alpha(x)(x-p)\in \partial\Omega$. Let us recall here that both $k$ and $n$ are defined almost everywhere as Alexandroff's theorem
states~\cite{alexandrov}.

Now, the Second Main Theorem of G.~Berck, A.~Bernig and C.~Vernicos in~\cite{berck_Bernig_Vernicos} ---
which encloses former results given by  B.~Colbois and P.~Verovic in~\cite{cpv} ---
asserts that in case $\partial \Omega $ is $C^{1,1}$ we have 
\begin{equation} \label{eq_entropy_coefficient_intro}
  \lim_{r \to +\infty} \frac{\vol_\Omega  B_\Omega(p,r)}{\sinh^{n-1}r}=\frac{1}{n-1} \cA_p(\Omega)\neq 0 
\end{equation}
and $\ent(\Omega)=n-1$ is a limit. Moreover, 
without any assumption on $\Omega $ we have $\lent(\Omega) \geq n-1$ whenever $\cA_p(\Omega) \neq 0$.

The third one --- which is also a rigidity result --- 
requires stronger assumptions about $\Omega$: it has to be divisible, meaning that it 
admits a compact quotient, and its Hilbert metric has to be hyperbolic in 
the sense of Gromov, which implies its boundary is $C^1$ and strictly convex by 
Y.~Benoist in~\cite{be}. Let us stress out that the Hilbert metric on such an $\Omega$ is 
the hyperbolic one if and only if 
$\Omega$ has a $C^{1,1}$ boundary, and that its volume entropy 
is positive since hyperbolicity implies the non-vanishing of the Cheeger constant
(see Theorem 1.5 in B.~Colbois and C.~Vernicos~\cite{ccv}). 
A result by M.~Crampon (see \cite{crampon}) states that 
for a divisible open bounded convex set $\Omega$ in $\R^n$ whose boundary is $C^1$
we have  $\ent(\Omega)\leq n-1$ with equality if and only if $\Omega$ is an ellipsoid.

In the present paper we link the volume entropy to another invariant associated
with a convex body, called the \textsl{approximability}. This name was introduced
by Schneider and Wieacker in~\cite{sw}. The approximability
measures in some sense how well a convex set can be approximated by polytopes. 
More precisely, let $N(\varepsilon,\Omega)$ be the smallest number
of vertices of a polytope whose Hausdorff distance
to $\Omega$ is less than $\varepsilon>0$. Then the lower and upper approximability of 
$\Omega$ are defined by
\begin{equation}
  \underline{a}(\Omega):=\liminf_{\varepsilon \to 0}\dfrac{\ln N(\varepsilon,\Omega)}{-\ln \varepsilon}\text{, and}\quad \overline{a}(\Omega):=\limsup_{\varepsilon \to 0}\dfrac{\ln N(\varepsilon,\Omega)}{-\ln \varepsilon}\text{.}
\end{equation}

The key inequality which is of interest in our work ---
obtained by Fejes-Toth~\cite{fejes} in dimension $2$ and by Bron\u ste\u\i n-Ivanov~\cite{bi}
in the general case --- asserts that for any bounded convex set in $\R^n$ 
the following upperbound on the upper approximability holds  $\overline{a}(\Omega)\leq (n-1)/{2}$.

Our main result is the following one
\begin{theo}[Main theorem]\label{MainTheo}
Given an open bounded convex set $\Omega$ in $\R^n$, we have
\begin{equation}
  \label{eqMainTheo}
  2\underline{a}(\Omega)\leq \lent(\Omega)\text{, and }\ 2\overline{a}(\Omega)\leq \Ent(\Omega)
\end{equation}
with equality for $n=2$ or $n=3$.
\end{theo}
The equality case in Theorem \ref{eqMainTheo} together with the uppebound for the upperapproximability
 imply the following  corollary.
\begin{cor}[Volume entropy upper bound conjecture]\label{UpperBoundConjecture}
  For any open bounded convex set $\Omega$ in $\R^2$ or $\R^3$ we have
$\Ent(\Omega)\leq n-1$.
\end{cor}

The equality case in this main theorem heavily relies on the study of polytopal Hilbert geometries. As it happens
we get an optimal control of the volume of metric balls in dimension two and three for in those two cases 
the number of edges of a polytope is bounded from above by the number of its vertices up to a multiplicative and 
an additive constant.
This does not hold in higher dimension, following McMullen's upper bound theorem~\cite{mcmullen,mcmullens}.

The second important results concerns the  two-dimensional case where we can prove that there are Hilbert geometries
with intermediate volume growth. 
\begin{theo}[Intermediate volume growth]\label{IntermediateVolumeGrowth}
  Let $f:\R^+\to \R^+$ be an increasing function that satisfies
$$
\liminf_{r\to +\infty} \frac{e^r}{f(r)}>0\text{.}
$$
Then there exist an open bounded convex set $\Omega$ in $\R^2$ and a point $o$ in $\Omega$, such that we have
\begin{equation}
  \label{eq:intermediateGrowthone}
  \liminf_{r\to +\infty} \frac{\vol_{\Omega}(B_{\Omega}(o,r))}{f(r)}>0 \text{ and } \limsup_{r\to +\infty} \frac{\vol_{\Omega}(B_{\Omega}(o,r))}{f(r)r^2}< +\infty\text{,} 
\end{equation}
and
\begin{equation}
  \label{eq:intermediateGrowthtwo}
  \begin{split}
    \lent(\Omega)= \liminf_{r\to +\infty} \frac{\ln f(r)}{r}\text{,}\\
 \Ent(\Omega)= \limsup_{r\to +\infty} \frac{\ln f(r)}{r}\text{.} 
  \end{split}
\end{equation}

In particular there are open bounded convex sets $\Omega\subset\R^2$ with 
\begin{itemize}
\item maximal volume entropy and zero centro-projective area,
\item zero volume entropy which are not polytopes.
\end{itemize}
\end{theo}
This theorem is a consequence of our method for proving the equality in dimension two in the Main theorem (see section \ref{lesinegalites})
and Schneider and Wieacker \cite{sw} results on the approximability in dimension two. The last statement follows
from our work \cite{ver09}, where we showed that polytopes have polynomial growth of order $r^2$ in dimension two.

The intermediate volume growth theorem allows us to settle in a quite definite way the question of whether the entropy is a limit or not.
\begin{cor}\label{cor:volentropyone}
The volume entropy is not a limit in general. More precisely,
for any $0\leq \alpha\leq \beta\leq 1$ there exist an open bounded convex set $\Omega$ in $\R^2$ such that 
we have
$$
\lent(\Omega)=\alpha, \qquad \Ent(\Omega)=\beta\text{.}
$$
\end{cor}

The equalities and inequalities also imply the following four new results,
\begin{cor} Given an open bounded convex set $\Omega$ in $\R^n$, we have 
  \begin{itemize}
  \item $d_H\leq\lent(\Omega)$, where $d_H$ is the Hausdorff dimension of
the set of farthest points of $\Omega$.
  \item if $n=2$ or $3$ then $\overline{a}(\Omega)$ is a projective invariant of $\Omega$
and $\Ent(\Omega)=\Ent(\Omega^*)$, where $\Omega^*$ is the polar dual of $\Omega$
  \item if $n=2$, then $\overline{a}(\Omega)\leq \frac1{3-d}$.
  \end{itemize}
\end{cor}

Section \ref{preliminaries} presents the various lemmas and notions needed in section \ref{lesinegalites} 
to prove
the main theorem, and in section \ref{growth} we present the proof of the intermediate volume growth theorem. 

\section{Preliminaries on Hilbert geometries and convex bodies}\label{preliminaries}

\subsection{Notations and definitions}

A \textsl{proper} open set in $\R^n$ is a set that does not contain a whole line.
A non-empty proper open convex set in $\R^n$ will be called a \textsl{proper convex domain}.
The closure of a bounded convex domain is usually called a \textsl{convex body}.

A Hilbert geometry
$(\Omega,d_\Omega)$ is a proper convex domain $\Omega$ in $\R^n$  endowed with
its Hilbert distance 
$d_\Omega$ defined as follows: for any distinct points $p$ and $q$ in $\Omega$,
the line passing through $p$ and $q$ meets the boundary $\partial \Omega$ of $\Omega$
at two points $a$ and $b$, such that $a$, $p$, $q$,
$b$ appear in that order on the line.  We denote by $[a,p,q,b]$ the cross ratio of $(a,p,q,b)$, i.e. 
$$
[a,p,q,b] = \frac{qa}{pa} \times \frac{pb}{qb} > 1,
$$
where for any two points $x$, $y$ in $\R^n$, $xy$ is their distance with 
respect to the standard Euclidean norm $||\cdot||$. Should $a$ or $b$ be at infinity, the corresponding ratio will
be considered equal to $1$.
Then we define 
$$
d_{\Omega}(p,q) = \frac{1}{2} \ln [a,p,q,b].
$$

Note that the invariance of the cross ratio by a projective map implies the invariance 
of $d_{\Omega}$ by such a map.

The proper convex domain $\Omega$ is also naturally endowed with
the  $C^0$ Finsler metric $F_\Omega$ defined as follows: 
given $p \in \Omega$ and $v \in T_{p}\Omega =\R^n$
with $v \neq 0$, the straight line passing through $p$ with direction vector 
$v$ meets $\partial \Omega$ at two points $p_\Omega^{+}$ and
$p_\Omega^{-}$ such that $p_\Omega^{+}-p_\Omega^{-}$ and $v$ have the same direction. 
Then let $t^+$ and $t^-$ be the two positive numbers such
that $p+t^+v=p_\Omega^{+}$ and $p-t^-v=p_\Omega^{-}$ (in other words
these numbers corresponds to the amounts of time needed to reach the boundary of $\Omega$ when starting
at $p$ with the velocities $v$ and $-v$, respectively). Then we define
$$
F_\Omega(p,v) = \frac{1}{2} \biggl(\frac{1}{t^+} + \frac{1}{t^-}\biggr) \quad \textrm{and} \quad F_\Omega(p , 0) = 0.
$$ 
Should $p_\Omega^{+}$ or
$p_\Omega^{-}$ be at infinity, then the corresponding ratio will be taken equal to $0$.


The Hilbert distance $d_\Omega$ is the length distance associated to 
$F_\Omega$. We shall denote by $B_\Omega(p,r)$ the metric ball of radius $r$
centred at the point $p\in \Omega$ and by $S_\Omega(p,r)$ the corresponding metric sphere.

Thanks to that Finsler metric, we can make use of two important Borel measures on
$\Omega$. 

The first one, which coincides with the Hausdorff measure associated to the metric 
space $(\Omega,d_\Omega)$, (see  example~5.5.13 in \cite{bbi}),  is the \textsl{Busemann} volume that we  will be denote by $\vol_\Omega$ and is
defined as follows.
Given any point  $p$ in $\Omega$, let $\beta_\Omega(p) = \{v \in \R^n ~|~ F_{\Omega}(p,v) < 1 \}$
be the open unit ball in
$T_{p}\Omega = \R^n$ with respect to the norm $F_{\Omega}(p,\cdot)$ and 
let $\omega_{n}$ be the Euclidean volume of the open unit ball of the standard Euclidean space
$\R^n$.
Then given any Borel set $A$ in $\Omega$, its  Busemann volume $\vol_\Omega$ is defined by
$$
\vol_\Omega(A) = \int_{A} \frac{\omega_{n}}{\lambda\bigl(\beta_\Omega(p)\bigr)} \ed{\lambda(p)}\text{,}
$$
where $\lambda$ denotes the standard Lebesgue measure on $\R^n$.

The second one, is the \textsl{Holmes-Thompson} volume on $\Omega$ that we will 
denote
by $\mu_{HT,\Omega}$. Given  any Borel set $A$ in $\Omega$ its Holmes-Thompson volume is
defined by
$$
\mu_{HT,\Omega}(A)= \int_{A} \frac{\lambda\bigl(\beta_\Omega^*(p)\bigr)}{\omega_n} \ed{\lambda(p)}\text{,}
$$
where $\beta_\Omega^*(p)$ is the polar dual of $\beta_\Omega(p)$.

We can actually consider a whole family of measures as follows.
Let ${\mathcal E}_n$ be the set of pointed proper open convex sets in $\R^n$. These are the pairs $(\omega,x)$ such that
$\omega$ is  a proper open convex set and $x$ is a point in $\omega$. 
We shall say that a function $f\colon {\mathcal E}_n\to \R$
is a \textsl{proper density} if it is positive and satisfies the three 
following properties
\begin{description}
\item[Continuity] with respect to the Hausdorff pointed topology on ${\mathcal E}_n$;
\item[Monotone decreasing] with respect to the inclusion, i.e., if $x\in \omega \subset \Omega$
then $f(\Omega,x)\leq f(\omega,x)$.
\item[Chain rule compatibility:] for any projective transformation $T$ one has
$$
f\bigl(T(\omega),T(x)\bigr) \text{Jac}(T,x)= f(\omega,x)\text{.} 
$$
\end{description}
We will say that $f$  is a \textsl{normalised proper density} if
$f(\omega,x)d\lambda(x)$ is the Riemannian volume when $\omega$ is an ellipsoid.
Let us denote by $\mathcal{PD}_n$  the set of proper densities over ${\mathcal E}_n$. 

A result of Benzecri~\cite{benzecri} states that the action of the group of
projective transformations on ${\mathcal E}_n$ is co-compact. Therefore,
for any pair $f,g$ in $\mathcal{PD}_n$, there exists a constant $C>0$ ($C\geq 1$ for the normalised ones)
such that for any $(\omega,x)\in {\mathcal E}_n$ one has
\begin{equation}
  \label{eqmeasures}
  \frac1C \leq \frac{f(\omega,x)}{g(\omega,x)}\leq C\text{.}
\end{equation}

Given a density $f$ in $\mathcal{PD}_n$ there is a natural Borel measure associated to any open bounded convex set
$\Omega$, denoted by  $\mu_{f,\Omega}$, and defined as follows: for any borel subset $A$ of  $\Omega$ we let  
$$
\mu_{f,\Omega}(A)= \int_{A} f(\Omega,p) \ed{\lambda(p)}\text{.}
$$

Integrating the inequalities (\ref{eqmeasures}) we obtain that  for any two proper densities $f$, $g$
in $\mathcal{PD}_n$, there exists a constant $C>0$ such that for any Borel set $A\subset  \Omega$ we have
\begin{equation}
  \label{eqmeasures2}
  \frac1C \mu_{g,\Omega }(A)\leq \mu_{f,\Omega}(A)\leq C \mu_{g,\Omega }(A)\text{.} 
\end{equation}
We shall call \textsl{proper measures with density} the family of
measures obtained in this way.

To a proper density $g\in \mathcal{PD}_{n-1}$ we can also associate a $(n-1)$-dimensional measure,
denoted by $\mu_{\cdot,g,\Omega}$,
on hypersurfaces in $\Omega$ as follows. Let $\Sigma$ be smooth a hypersurface,
and consider for a point $p$ in the hypersurface $\Sigma$ its tangent hyperplane $H(p)$,
then the measure will be given by 
\begin{equation}
  \label{eqarea}
  \frac{d\mu_{\Sigma,g,\Omega}}{d\sigma}(p) = \frac{d\mu_{g,{\Omega }\cap H(p)}}{d\sigma}(p)\text{.}
\end{equation}
Where $\sigma$ denotes the Hausdorff $n-1$-dimensional measure associated with
the standard Euclidean distance.
In section \ref{lesinegalites} 
we will simply denote respectively  by $\vol_{n-1,\Omega}$ and $\Area_\Omega$ the 
$(n-1)$-dimensional measure associated respectively with the Holmes-Thompson and the Busemann measures.

Let now $\mu_{f,\Omega}$ be  a proper measure with density over $\Omega$,
then the volume entropies of $\Omega$ is defined by
\begin{multline}
  \label{eqvolentropie}
 \lent(\Omega)=\liminf_{r\to +\infty} \dfrac{\ln \mu_{f,\Omega}\bigl(B_\Omega(p,r)\bigr)}{r}\text{, and }\\ \Ent(\Omega)=\limsup_{r\to +\infty} \dfrac{\ln \mu_{f,\Omega}\bigl(B_\Omega(p,r)\bigr)}{r} \text{.}
\end{multline}
These number do not depend on either $f$ nor $p$, and are  equal to the 
spherical entropies (see Theorem 2.14 \cite{berck_Bernig_Vernicos}):
\begin{multline}\label{sphericalentropy}
  \lent(\Omega)=\liminf_{r\to +\infty} \dfrac{\ln \Area_{\Omega}\bigl(S_\Omega(p,r)\bigr)}{r}\text{, and }\\
  \Ent(\Omega)=\limsup_{r\to +\infty} \dfrac{\ln \Area_{\Omega}\bigl(S_\Omega(p,r)\bigr)}{r} \text{.}
\end{multline}

\subsection{Properties of the Holmes-Thompson and the Busemann measures}
\begin{quote}
  We recall some properties of the Holmes-Thompson and the Busemann measure.
\end{quote}

\begin{lema}[Monotonicity of the Holmes-Thompson measure]\label{l:htmonotonicity}
Let $(\Omega, d_\Omega)$ be a Hilbert geometry in $\R^n$. The Holmes-Thompson area measure is
monotonic on the set of convex bodies in $\Omega$, that is, 
  for any $K_1$ and $K_2$ pair of convex bodies in $\Omega$, such that $K_1\subset K_2$ on has
  \begin{equation}\label{htmonotonicity}
    \vol_{n-1,\Omega}(\partial K_1)\leq \vol_{n-1,\Omega}(\partial K_2)\text{.}
  \end{equation}
\end{lema}
\begin{proof}
     If $\partial\Omega$ is $C^2$ with everywhere positive Gaussian curvature then the
tangent unit spheres of the Finsler metric are quadratically convex.

According to \'Alvarez Paiva and Fernandes \cite[theorem 1.1 and remark 2]{alvarez_fernandez} there exists a Crofton
formula for the Holmes-Thompson area, from which the inequality (\ref{htmonotonicity}) follows.

   Such smooth convex bodies are dense in the set of all convex bodies
for the Hausdorff topology. By approximation, it follows that inequality (\ref{htmonotonicity}) is valid for any $\Omega$.
\end{proof}

Lemma~\ref{l:htmonotonicity} associated with 
the Blaschke-Santalo inequality and the inequality~(\ref{eqmeasures2})
 immediately implies the following result (see also \cite[Lemma~2.12]{berck_Bernig_Vernicos}).

\begin{lema}[Rough monotonicity of the Busemann measure]\label{rmonotone}
Let $(\Omega, d_\Omega)$ be a Hilbert geometry, and let $p$ be a point in $\Omega$.
There exists a monotonic function $f_\Omega$ and a constant $C_n<1$ such that for all $r>0$
\begin{equation} \label{eq_comparison_holmes_thompson}
 C_n f_\Omega(r) \leq \Area_\Omega(S_\Omega(p,r)) \leq  f_\Omega(r).
\end{equation}
$f_\Omega(r)$ is the Holmes-Thompson area of the sphere $S_\Omega(p,r)$
\end{lema}

Let us finish by recalling one last statement also proved in \cite[Lemma~2.13]{berck_Bernig_Vernicos}.
\begin{lema}[Co-area inequalities]\label{coarea}
For all $r>0$
\begin{displaymath}
 \frac{1}{2}\frac{\omega_n}{\omega_{n-1}} \Area_\Omega(S_\Omega(p,r)) \leq \frac{\partial}{\partial r}\vol_\Omega(B_\Omega(p,r)) \leq \frac{n}{2}\frac{\omega_n}{\omega_{n-1}} \Area_\Omega(S_\Omega(p,r)).
\end{displaymath}
\end{lema}

\subsection{Upper bound on the area of triangles }

\begin{quote}
  In this section we bound from above independently of the two-dimensional Hilbert
geometries the area of affine triangles which are subset of a metric ball, when
one the vertexes is the centre of that ball. We also give a lower bound
on the length of some metric segments, when their vertexes go to the boundary
of the Hilbert geometry.
\end{quote}

\begin{lema}\label{l:trianglearea}
Let $(\Omega,d_\Omega)$ be a two-dimensional Hilbert geometry. Then there
exists a constant $C$ independent of $\Omega$, such that, for any point $o$ in $\Omega$
and any pair of points $p_\rho$ and $q_\rho$ in the metric ball $B_{\Omega}(o,\rho)$,
the area of the affine triangle $(op_\rho q_\rho)$ is less than $C\rho^2$.   
\end{lema}

\begin{proof}
Given $p_\rho$ and $q_\rho$ in $B_\Omega(o,\rho)$, let $p$ and $q$ be the intersections
of the boundary $\partial\Omega$ with the half lines $[o,p_\rho)$ and $[o,q_\rho)$ respectively.
Let $p'$ and $q'$ be, respectively, the intersections of 
the half lines $[p_\rho,o)$ and $[q_\rho,o)$ with
the  boundary $\partial\Omega$. 

\begin{figure}[h]
  \centering
  \includegraphics{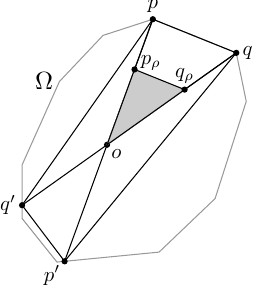}
  \caption{The area of the triangle $(op_\rho,q_\rho )$ is bounded by $C\rho^2$.}
\end{figure}

Then the volume of the triangle $(op_\rho q_\rho)$ with respect to the Hilbert geometry of 
$\Omega$ is less than or equal to its volume with respect to the Hilbert geometry of the
quadrilateral $(pqp'q')$. However, the distances of $p_\rho$ and $q_\rho$ from $o$ 
remain the same in both Hilbert geometries. 

Up to a change of chart, we can suppose
that this quadrilateral is actually a square. This allows us
to use Theorem 1 from~\cite{ver2011} which states that 
the Hilbert geometry of
the square is bi-lipschitz to the product of the Hilbert geometries of its sides, 
using the identity as a map. 
In other words it is bi-lipschitz to the Euclidean plane, with a lipschitz constant
equal to $C_0>1$, independent of our initial conditions.

Therefore our affine triangle is inside a Euclidean  disc of radius $C_0\rho$, 
which implies that its area with respect to the Hilbert geometry of $\Omega$ is less than $C_0^4 \times \pi \times\rho^2$.
\end{proof}

To prove that the volume entropy is bounded from below by the approximabilty we will need
to bound from below the length of certain segments in a given Hilbert geometry $\Omega$. To do so will
compare their length in the initial convex with their length in a convex projectively equivalent to
a triangle, and  containing the initial convex $\Omega$.

Let us make this precise. Consider four points $a$, $b$, $c$ and $d$ in the
Euclidean plane $(\R^2,\langle\cdot\rangle)$ such that $\mathcal{Q}=(abcd)$ is a convex quadrilateral.
We assume that the scalar products $\langle\vec{ab},\vec{bc}\rangle$ and $\langle\vec{bc},\vec{cd}\rangle$ 
are positive and we
let $q$ be the intersection point between the straight lines $(ab)$ and $(cd)$.

Suppose that $\Omega$ is a convex domain such that the segments
 $[a,b]$, $[b,c]$ and $[c,d]$ belong to its boundary.

Given $p$ a point in the convex domain $\Omega$ we denote by $p'$ 
the intersection between the straight line $(pq)$ and the segment $[b,c]$,
and  we define $s={bp'}/{bc}.$
 
We then denote by $[b(r),c(r)]$ the image of the segment $[b,c]$ under the dilation centred at $p$ 
with ratio $0<\tanh(r)<1$. The image of the segment
$[b,c]$ under the dilation centred at $q$ sending $p'$ on $p$ will
be denoted by $[B,C]$.

\begin{figure}[h]
  \centering
  \includegraphics[scale=0.7]{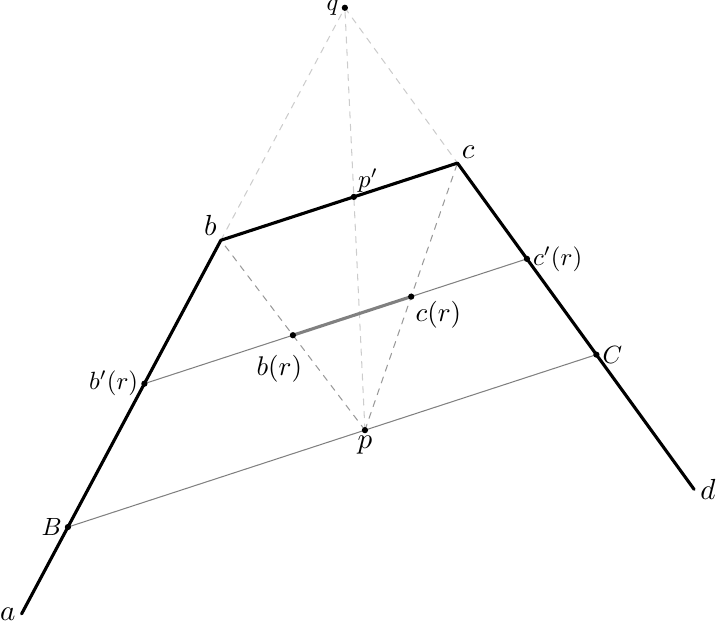}
  \caption{Distance estimate of Claim \ref{lessegments}}
\end{figure}

\begin{claim}
  \label{lessegments}
The following inequality is satisfied under the above assumption:
\begin{equation}\label{minorationlg}
  \begin{split}
     d_\Omega\bigl(b(r),c(r)\bigr) \geq \frac{1}{2}\ln\biggl(\frac{bc}{s\cdot BC}&\frac{\tanh(r)}{1-\tanh(r)}+1\biggr) \\ &+\frac{1}{2} \ln\biggl(\frac{bc}{(1-s)\cdot BC}\frac{\tanh(r)}{1-\tanh(r)}+1\biggr)\text{.}
  \end{split}
\end{equation}
\end{claim}
\begin{proof}
  Straightforward computation, using the fact that the convex domain $\Omega$ is inside the convex $Q$ obtained as the
intersection of the half planes defined by the lines $(ab)$, $(bc)$ and $(cd)$, and therefore
$$
d_\Omega\bigl(b(r0,c(r)\bigr)\geq d_Q\bigl(b(r),c(r)\bigr).
$$

Let  $b'(r)$ be the intersection of the lines $(ab)$ and $\bigl(b(r)c(r)\bigr)$,
and let $c'(r)$ be the intersection of the lines $(cd)$ and $\bigl(b(r)c(r)\bigr)$.
Then we have 
$$
d_Q\bigl(b(r0,c(r)\bigr)=\frac{1}{2}\ln \biggl(\frac{b(r)c'(r)}{c(r)c'(r)}\cdot \frac{c(r)b'(r)}{b(r)b'(r)}\biggr).
$$
Let us focus on the first ratio. On the one hand $b(r)c'(r)=b(r)c(r)+c(r)c'(r)$, and on the second
hand following Thales's theorem
\begin{equation}
  \begin{split}
    &b(r)c(r)=\tanh(r)bc\\
    &c(r)c'(r)=(1-\tanh(r))pC.
  \end{split}
\end{equation}
But $pC=BC\cdot(p'c/bc)=(1-s)BC$, and therefore we obtain
$$
\ln \biggl(\frac{b(r)c'(r)}{c(r)c'(r)}\biggr) = \ln\biggl(\frac{bc}{(1-s)\cdot BC}\frac{\tanh(r)}{1-\tanh(r)}+1\biggr).
$$
The second ratio is treated in the same way.
\end{proof}

\subsection{Intrinsic and extrinsic Hausdorff topologies of Hilbert Geometries}

\begin{quote}
We describe the link between the Hausdorff topology induced by an Euclidean metric
with the Hausdorff topology induced by the Hilbert metric on compact subset of an open
convex set. 
\end{quote}

We recall that the Lowner ellipsoid of a compact set, is the ellipsoid with least volume containing
that set. In this section we will suppose, without loss of generality, that $\Omega$ is a bounded open convex set, whose Lowner
ellipsoid $\mathcal{E}$ is the Euclidean unit ball and $o$ is the center of that ball.
It is a standard result that $(1/n)\mathcal{E}$ is then contained in $\Omega$, i.e., we have the following sequence
of inclusions
\begin{equation}
  \frac{1}{n}\mathcal{E}\subset\Omega\subset\mathcal{E}
\end{equation}

\begin{defi}[Asymptotic ball and sphere]
  We call \textsl{asymptotic ball} of radius $R$ centred at $o$ the image
of $\Omega$ by the dilation of ratio $\tanh R$ centred at $o$, and we denote it by $\asb(o,R)$.
The image of the boundary $\partial \Omega$ by the same dilation will be called the \textsl{asymptotic sphere}
of radius $R$ centred at $o$ and denoted by $AsS(o,R)$.
\end{defi}

Recall that the Hausdorff distance is a distance between non empty compact subsets in a metric space. We shall 
use both the Euclidean and Hilbert distance and we will use the terminology 
\textsl{Haus\-dorff-Euclidean} and \textsl{Haus\-dorff-Hilbert} to distinguish both cases.

We would like to relate the Hausdorff-Hilbert neighbourhoods of the asymptotic ball $\asb(o,R)$ with
its Hausdorff-Euclidean neighbourhoods.

\begin{prop}\label{keyprop}
Let $\Omega$ be a convex domain and let $o$ be the centre of its Lowner ellipsoid, which is supposed to be the unit Euclidean ball.
  \begin{enumerate}
  \item  The $(1-\tanh(R))/2n$-Haus\-dorff-Euclidean neighborhood of the asymptotic ball $\asb(o,R)$ is contained in
its $\bigl((\ln 3)/2\bigr)$-Haus\-dorff-Hilbert neighborhood.
\item For any $K>0$, the $K$-Haus\-dorff-Hilbert neighborhood of the asymptotic ball $\asb(o,R)$ is contained in its $\bigl(1-\tanh(R)\bigr)$-Hausdorff-Euclidean neighborhood.
  \end{enumerate}
\end{prop}
\begin{proof}
  For any point $p\in \partial\Omega$ on the boundary of $\Omega$, and for $0<t<1$ let $\varphi_t(p)=o+t\cdot\overrightarrow{op}$. This map sends
$\partial\Omega$ bijectively on the asymptotic sphere  $AsS(o,\arctanh t)$ centred at $o$ with radius $\arctanh t$.

\medskip
\noindent\textsl{Proof of part (i) of the Proposition:}

Any point of a compact set in the $(1-\tanh(R))/2n$-Hausdorff-Euclidean neighborhood of $\asb(o,R)$, either lies inside
$\asb(o,R)$, or is contained in an Euclidean ball of radius $(1-\tanh(R))/2n$ centred on a point of $\asb(o,R)$.

We recall that the ball of radius $1/n$ is a subset of $\Omega$, and thus so is the ball of radius $1/2n$, that is
$$
\frac{1}{2n}\mathcal{E}\subset\frac{1}{n}\mathcal{E}\subset\Omega\text{.}
$$ 
Let $p\in \partial\Omega$ be a point on the boundary. 
By convexity, the interior of $K(p)$ the convex hull 
of $p$ and 
$1/n\mathcal{E}$ is a subset of $\Omega$ --- 
it is the projection of a cone of basis $1/n\mathcal{E}$. 
Hence $\mathcal{E}_{p,\alpha}$, the
image of $1/n\mathcal{E}$ by the dilation of ratio $0<\alpha<1$ centred at $p$, lies in the "cone" $K(p)$.
The set $\mathcal{E}_{p,\alpha}$ is therefore an Euclidean ball of radius $\alpha/n$ centred at $\varphi_{1-\alpha}(p)$, and it is a subset of $\Omega$.

A point in the Euclidean ball of radius
$\alpha/2n$ centred at $\varphi_{1-\alpha}(p)$ is at a distance less or equal to $1/2\ln 3$ from $\varphi_{1-\alpha}(p)$ with respect to the Hilbert distance of $\mathcal{E}_{p,\alpha}$.

Now a standard comparison arguments states that for any two points $x$ and $y$ in $\mathcal{E}_{p,\alpha}\subset\Omega$
the following inequality occurs
$$
d_\Omega(x,y)\leq d_{\mathcal{E}_{p,\alpha}}(x,y)\text{.}
$$
From this inequality it follows that 
any point in the Euclidean ball of radius $\alpha/2n$ centred at $\varphi_{1-\alpha}(p)$ 
is in the Hilbert metric ball centred at $\varphi_{1-\alpha}(p)$ of radius $1/2\ln3$.

Now for any $1\geq\alpha>1-\tanh R$, the Euclidean ball of radius $\alpha/2n$ contains the Euclidean ball of radius $(1-\tanh R)/2n$.

This implies that for any point $x$ in the asymptotic ball $\asb(o,R)$, the Euclidean ball of radius $(1-\tanh R)/2n$ centred at $x$ is contained in the Hilbert ball of radius $1/2\ln 3$ centred at the $x$, which allows us
to obtain the first part of our claim. 
  
\begin{figure}[hbtp]
  \centering
  \includegraphics{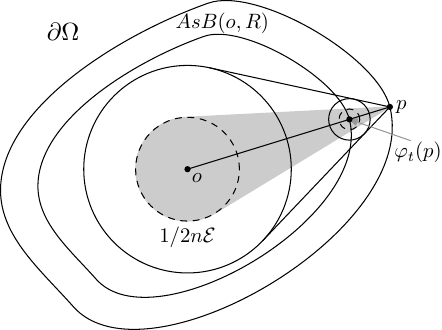}
  \caption{Illustration of  Proposition \ref{keyprop}'s proof}
\end{figure}

\bigskip
\noindent\textsl{Proof of part (ii) of the Proposition:}
This follows from the fact that under our assumptions, $\Omega$ itself is in the $(1-\tanh R)$ Hausdorff-Euclidean
neighborhood of the asymptotic ball $\asb(o,R)$.
\end{proof}
 
\begin{cor}\label{keycor}
Let $\Omega$ be a convex domain and let $o$ be the centre of its Lowner ellipsoid, which is supposed to be the unit Euclidean ball.
  \begin{enumerate}
  \item  The $(1-\tanh(R+\ln 2))/2n$-Hausdorff-Euclidean neighborhood of $B(o,R)$ is contained in
its $\ln \bigl(3(n+1)\bigr)$-Hausdorff-Hilbert neighborhood.
\item For any $K>0$, the $K$-Hausdorff-Hilbert neighborhood of $B(o,R)$ is contained in its $\Bigl(1-\tanh\bigl(R+K-\ln(n+1)\bigr)\Bigr)$-Hausdorff-Euclidean neighborhood.
  \end{enumerate}
\end{cor}

The proof of this corrollary is a straightforward consequence of the following lemma
applied to the conclusion of the Proposition~\ref{keyprop}.
\begin{lema}\label{keyrmk}
Let $\Omega$ be a convex domain, and suppose that $o$ is a point in the interior
of $\Omega$ such that the unit Euclidean open ball centred at $o$ contains $\Omega$,
and $\Omega$ contains the Euclidean closed ball centred at $o$ of radius $1/(2n)$. 
Then we have
   \begin{equation}\label{inclusions}
    \begin{split}
      B(o,R)\subset\asb(o,R+\ln 2)\text{, and}\\
      \asb(o,R)\subset B\bigl(o,R+\ln(n+1)\bigr). 
    \end{split}
  \end{equation}
\end{lema}
This lemma is a refinement of a result of \cite{cpv} in our case.

\begin{proof}[Proof of Lemma \ref{keyrmk}]
Let $x$ be a point on the boundary $\partial\Omega$ of $\Omega$, and let $x^*$ be the second intersection of the straight line $(ox)$ with $\partial\Omega$.
Then our assumption implies the next two inequalities.
\begin{equation}\label{insidelowner}
  \begin{split}
    1/2n< xo\leq1\\
 1/2n< ox^*\leq1 
  \end{split}
\end{equation}

Actually the first inclusion is always true. Indeed suppose $y$ is on the half line $[ox)$ such that
$d_\Omega(o,y)\leq R$ which in other words implies that we have
$$
\frac{ox}{yx}\frac{yx^*}{ox^*}\leq e^{2R}
$$
therefore
$$
ox\leq e^{2R} \frac{ox^*}{yx^*} (ox-oy) \leq e^{2R} (ox-oy)
$$
which implies in turn that
$$
oy\leq \frac{e^{2R}-1}{e^{2R}} ox \leq \bigl(1-e^{-2R}\bigr) ox\leq \tanh(R+\ln 2) ox\text{.}
$$

Now regarding the second inclusion: consider $y$ a point on the half line $[ox)$  such
that $oy\leq \tanh(R) ox$. On the one hand we have
$$
\frac{ox}{yx}=\frac{ox}{ox-oy}\leq \frac1{1-\tanh(R)}=\frac{e^{2R}+1}2\text{.}
$$
and, on the other hand thanks to the inequalities (\ref{insidelowner}) we get
\begin{equation}
  \label{eqincone}
  \frac{yx^*}{ox^*}\leq\frac{ox+ox^*}{ox^*}\leq 1 +\frac{ox}{ox^*}\leq1 + 2n\text{,}
\end{equation}
which implies that
\begin{equation}
  \label{eqinctwo}
\frac{ox}{yx}\frac{yx^*}{ox^*}\leq\frac{e^{2R}+1}{2}(1+2n)\leq (1+2n)e^{2R} \leq(1+n)^2e^{2R}\text{.}
\end{equation}
The conclusion follows.
\end{proof}


\subsection{Distance function to a sphere  in a Hilbert geometry}\label{Adistanboule}
This section is an adaptation in the realm of Hilbert geometries of a result concerning the spheres in a Minkowski space
provided to the author by A. Thompson~\cite{essay17}.

Let us first start by recalling the following important fact regarding the
distance of a point to a geodesic in a Hilbert geometry (see Busemann~\cite{busemann}, chapter II, section 18, page 109):

\begin{prop}\label{peakless}
  Let $(\Omega,d_\Omega)$ be a Hilbert Geometry. 
The distance function of a straight geodesic (that is given by an affine line) 
to a point is a peakless
function, i.e., if $\gamma\colon [t_1,t_2]\to \Omega$ is a geodesic segment,
then for any $x\in \Omega$ and $ t_1\leq s\leq t_2$ one has
$$
d_\Omega\bigl(x,\gamma(s)\bigr)\leq \max\biggl\{d_\Omega\bigl(x,\gamma(t_1)\bigr),d_\Omega\bigl(x,\gamma(t_2)\bigr)\biggr\}\text{.} 
$$
\end{prop}

Let us now turn our attention to metric spheres in a two dimensional Hilbert
geometry.

\begin{prop}\label{spheremonotonic}
  Let $(\Omega,d_\Omega)$ be a two dimensional Hilbert Geometry. 
Suppose $o$ is a point of $\Omega$, and  $p$ and $q$ are two points on 
the intersection of the metric
sphere $S(o,R)$ centred at $o$ and radius $R$ with a line passing by $o$.
If $C$ denotes one of the arcs of the sphere $S(o,R)$ from $p$ to $q$,
then for any point $p'$ on the half line $[o,p)$, the function
$\varphi(x)=d_\Omega(p',x)$ is monotonic on $C$.
\end{prop}
\begin{proof}
  Let $p,x,y,q$ be points on that order on $C$. We have to show
that 
$$
d_\Omega(p',x)\leq d_\Omega(p',y)\text{.}
$$

\begin{figure}[h]
$$  \includegraphics[scale=.7]{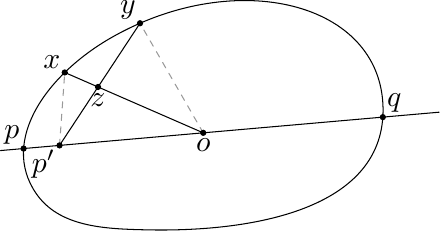}
  \includegraphics[scale=.7]{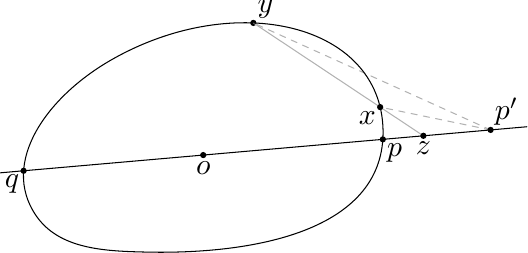}$$
 \caption{Monotonicity of the distance of a point to a sphere}
\end{figure}
Suppose first that that the line segments $[o,x]$ and $[p',y]$ 
intersects at a point $z$.
Hence we have
\begin{eqnarray*}
  d_\Omega(o,x)+d_\Omega(p',y)&=&\bigl(d_\Omega(o,z)+d_\Omega(z,x)\bigr) + \bigl(d_\Omega(p',z)+d_\Omega(z,y)\bigr)\\
                 &=& \bigl(d_\Omega(p',z)+d_\Omega(z,x)\bigr) + \bigl(d_\Omega(o,z)+d_\Omega(z,y)\bigr)\\
                 &\geq&  d_\Omega(p',x)+d_\Omega(o,y)\text{.} 
\end{eqnarray*}
now, as $d_\Omega(o,y)=d_\Omega(o,x)=R$, the result follows.

\medskip

Suppose now that $[o,x]$ and $[p',y]$ do not intersect, which implies
that $p'$ is outside the ball $B(o,R)$. Then the line $(yx)$
intersects $(op)$ at $z$. Because $x$ and $y$ lie on the sphere of radius $R$,
$d_\Omega(o,z)>R$.
Also, as $p$ is one of the nearest points to $p'$ on $C$, we have
$d_\Omega(p',z)\leq d_\Omega(p',p)\leq d_\Omega(p',y)$
Hence if apply the proposition \ref{peakless} to the segment
$[z,y]$ and $p'$, as $x\in [z,y]$ we get
$$
d_\Omega(p',x)\leq \max\bigl\{d_\Omega(p',z),d_\Omega(p',y) \bigr\} = d_\Omega(p',y). 
$$

\end{proof}

\section{Volume entropy and approximability}\label{lesinegalites}

\begin{quote}
 This section is devoted to the proof of the main theorem. This is done in two steps.
The first step consists in bounding the entropy from above in dimension 2 and 3 by the approximability
thanks to the study of the volume growth in polytopes. 
The second step is to bound from below the
entropy. 
This is done by exhibiting a separated subset of the Hilbert geometry 
whose growth
is bigger than the approximability. 
We conclude this section with the various corollaries implied.
\end{quote}

\begin{theo}\label{pdmajorant}
Let $\Omega$ be a bounded convex domain in $\R^2$ or $\R^3$.
The double of the approximabilities of $\Omega$ are bigger than the volume entropies, \textsl{i.e.},
$$
\lent(\Omega)\leq 2\underline{a}(\Omega)\text{, and } \Ent(\Omega)\leq 2\overline{a}(\Omega)\text{.}
$$  
\end{theo}

The proof of this theorem relies on the following stronger statement which is a sort of uniform bound
on the volume of metric balls and metric spheres in a polytopal Hilbert geometry.  The key fact is
that this bound depends, in a coarse sense, linearly on the number of verticies of the polytope.

\begin{theo}\label{growthpolytope}
Let $n=2$ or $n=3$. 
There are affine maps $a_n,b_n$ from $\R\to \R$ and  polynomials $q_n,p_{n-1}$ of degree $n$ and $n-1$ such that for any 
 open convex polytope ${\mathcal P}_N$ with $N$ vertices inside the unit Euclidean ball 
of $\R^n$ and  containing the ball of radius $1/2n$, one has
\begin{equation}\label{growthcontrol}
  \begin{array}{rcl}
\vol_{n-1,{\mathcal P}_N}S_{{\mathcal P}_N}(o,R)&\leq& a_n(N)p_{n-1}(R)\\
    \vol_{{\mathcal P}_N}B_{{\mathcal P}_N}(o,R)&\leq& b_n(N)q_n(R)\text{.} 
  \end{array}
\end{equation}
The same result holds for the asymptotic balls.
\end{theo}

Let us stress out that our method also yields a control in terms of the vertices in higher dimension as well, 
using the so called upper bound conjecture proved by McMullen \cite{mcmullen, mcmullens},
but alas a polynomial of degree strictly bigger than $1$ replaces the affine functions $a_n$ and $b_n$. 
This is why we can't state the equality in the Main Theorem in higher dimensions. 

Notice that this theorem is still valid if we replace the Hausdorff measures by
any measures defined by a  pair of  proper densities $f\in \mathcal{PD}_n$ and $g\in \mathcal{PD}_{n-1}$.
The change of measures will only impact the values of the constants. 

\begin{proof}[Proof of theorem \ref{growthpolytope}]
We will have to deal with the dimension two and the
dimension three separately, even if both cases follow the same main steps.

The first step of our proof 
consists in proving the first inequality of (\ref{growthcontrol}) for the Holmes-Thompson measure 
and for an asymptotic sphere. The uniform inclusion of metric balls into asymptotic balls~(\ref{inclusions}) 
imply then the result for the metric spheres thanks to the monotonicity of the Holmes-Thompson measure lemma~\ref{l:htmonotonicity}.

The second step is an integration using the co-area inequality~(\ref{withcoarea}), wich allows us to get the
second inequality of (\ref{growthcontrol}) for metric balls with respect to the Buseman measure.

Let us now make all this more precise.
We fix a Polytope ${\mathcal P}_N$ with $N$ verticies and for any real $R>0$ we
let $P_R$ be the asymptotic ball of radius $R$ centred at $o$, and let  $\partial P_R$ be the associated asymptotic sphere. 
We also introduce the constant  $c_n=\ln(n+1)$.

\bigskip

\noindent\textsl{Two dimensional case:} 
The idea is to found an upper bound on the length of each eadge of the asymptotic sphere $\partial P_R$,
depending only on $R$.

To do so, we can use the fact that eadge edge belongs to the triangle defined by joining its extremities to
the point $o$. Hence, thanks to the triangular inequality its length is less than the sum of these
two other segments. However, using the second inclusion (\ref{inclusions}) of lemma \ref{keyrmk},
we know that the asymptotic ball $P_R$ is inside the Hilbert ball of radius $R+c_2$ centred at $o$ of
the convex polygon $\mathcal{P}_N$. Hence the length of each edge is less than $2\cdot(R+c_2)$.
Therefore the length of the polygon $\partial P_R$ is less than $N\cdot 2\cdot(R+c_2)$.

Following the first inclusion~(\ref{inclusions}) of Lemma~\ref{keyrmk},
the metric ball of radius $r$ centred at $o$ is a subset of the asymptotic ball of radius $r+\ln 2$
centred at $o$. Therefore, we can 
 use the monotonicity of the Holmes-Thompson length (see Lemma~\ref{l:htmonotonicity}) to get 
for all $r>0$,
\begin{equation}\label{beforecoarea}
  \text{length}_{\mathcal{P}_N}(S_{\mathcal{P}_N}(o,r))\leq 
\text{length}_{\mathcal{P}_N}(\partial P_{r+\ln(2)})\leq N\cdot 2(r+\ln 2+c_2)\text{.}
\end{equation}
Now using the co-area inequality of Lemma~\ref{coarea}, taking into
account that the Busemann length is equal to the Holmes-Thompson length one gets
\begin{equation}
  \label{withcoarea}
\frac {\partial}{\partial r} \vol_{\mathcal{P}_N}(B_{\mathcal{P}_N}(o,r)) \leq  \frac{\pi}{4}\cdot N\cdot 2(r+\ln 2+c_2)\text{.}
\end{equation}
Hence, integrating the  inequality~(\ref{withcoarea}) over the interval $[0,R]$,
we finally obtain the following inequality for the ball of radius $R>0$
\begin{equation}\label{aftercoarea}
  \vol_{\mathcal{P}_N}(B_{\mathcal{P}_N}(o,R))\leq\frac{\pi}{4}\cdot N\cdot (R^2+2(\ln 2+c_2)R)\text{.}
\end{equation}
The inequalities~(\ref{beforecoarea}) and (\ref{aftercoarea}) are the expected
results in dimension two.

\bigskip

\noindent\textsl{Three dimensional case:} 
Once again the idea is to found an upper bound on the area of faces of the asymptotic sphere $\partial P_R$.
Alas, contrary to the two dimensional cases, there is not a unique type of faces, and is therefore
pointless to look for an upper bound depending only on the radius $R$.

However, each face can be seen as the basis of a pyramid with apex the point $o$. All other faces
are then triangles, whose areas can be bounded thanks to the lemma~ \ref{l:trianglearea}.
The analog of the triangle inequality is available in the form of the  minimality of the  Holmes-Thompson area (see Berck~\cite{berck}). In other words, the Holmes-Thompson area of each face of $\partial P_R$  
is less than the sum of the Holmes-Thompson areas of 
the triangles obtained as the  convex hull of $o$ and an edge of the given face of $\partial P_R$. Let us call $\mathcal{T}_o$
such a triangle (the subscript $o$  is to stress the fact that the point $o$ is one of its verticies)

\smallskip
To bound the area of the triangle $\mathcal{T}_o$ it suffices to focus on the intersection of the polytope 
 ${\mathcal P}_N$ with the affine plane containing the triangle $\mathcal{T}_o$. This is a polygon $\tilde P$,
to which we can apply the
Lemma \ref{l:trianglearea} which bounds from above the area of a two dimensional triangles inside a
 metric ball centred on one of its vertex.  Which is exactly the situation of our triangle  $\mathcal{T}_o$
with respect to the Hilbert geometry associated to the polygon $\tilde P$. Indeed it is included in
the asymptotic ball of radius $R$, and again thanks to lemma \ref{keyrmk} we know that
it is inside the metric ball of radius $R+c_3$ with respect to Hilbert geometry of $\mathcal{P}_N\in \R^3$.
As $\tilde P$ is a plane section of $\mathcal{P}_N\in \R^3$, this still holds for $\mathcal{T}_o$ seen
as a subset of $\tilde P$. Hence Lemma \ref{l:trianglearea}
implies that the area of the triangle $\mathcal{T}_o$ is less than $C(R+c_3)^2$, 
for some constant $C>1$ independent of $R$.

\smallskip 

Therefore, if $e(N)$ is the number of edges of $\mathcal{P}_N$,
the area of the asymptotic sphere $\partial P_R$ is less than $2e(N)C(R+c_3)^2$. 

Let $f(N)$ be the number of faces of $\mathcal{P}_N$ and let us 
recall Euler's formula:
$$
N-e(N)+f(N)=2\text{.}
$$
Each face being surrounded by at 
least three edges and each edge belonging to two
faces, one has the classical inequality  
(where equality is obtained in a simplex),  
$$
3f(N)\leq 2e(N)\text{.} 
$$
Combining the previous two inequalities we get a linear upper 
bound of the number of edges by the number of vertexes as follows:
$$
2\leq N-(1/3)e(N) \Rightarrow e(N)\leq 3N-6\text{.} 
$$
Hence the area of of the asymptotic sphere $\partial P_R$ is less than $(3N-6)\cdot 2C\cdot(R+c_3)^2$.

We can now conclude almost as in the two dimensional case.
Following the first inclusion~(\ref{inclusions}) of Lemma~\ref{keyrmk},
the metric ball of radius $r$ centred at $o$ is a subset of the asymptotic ball of radius $r+\ln 2$
centred at $o$. 
Therefore, we can 
 use the monotonicity of the Holmes-Thompson area measure (see Lemma~\ref{l:htmonotonicity}) to get 
for all $r>0$,
\begin{equation}\label{majoairesphere-HT}
  \vol_{2,\mathcal{P}_N}(S_{\mathcal{P}_N}(o,r)) \leq\vol_{2,\mathcal{P}_N}(\partial P_{r+\ln2})\leq (3N-6)2C(r+\ln2+c_3)^2\text{.}
\end{equation}
Notice that this inequality~(\ref{majoairesphere-HT}) corresponds to the first part of the 
inequality~(\ref{growthcontrol}).

The rough monotonicity  of the Busemann measure (see the right hand side of the inequality~(\ref{eq_comparison_holmes_thompson}) in Lemma~\ref{rmonotone}) states that the Busemann area is smaller that the Holmes-Thompson
one, hence  combined with the inequality (\ref{majoairesphere-HT}) above, we get that for all $r>0$
\begin{equation}\label{majoairesphere-B}
  \Area_{\mathcal{P}_N}(S_{\mathcal{P}_N}(o,r))\leq(3N-6)\cdot 2C\cdot(r+\ln2+c_3)^2\text{.}
\end{equation}
Taking into account the 
co-area inequality (see Lemma~\ref{coarea}) in conjunction with the inquality~(\ref{majoairesphere-B})
leads to the following differential inequality
\begin{equation}
  \frac{\partial}{\partial r}\vol_{\mathcal{P}_N}(B_{\mathcal{P}_N}(o,r))\leq2\cdot (3N-6)\cdot 2C\cdot(r+\ln2+c_3)^2\text{,}
\end{equation}
which we can integrate over the interval  $[0,R]$ to finally obtain that for all $R>0$
\begin{equation}
  \vol_{\mathcal{P}_N}(B_{\mathcal{P}_N}(o,R))\leq2\cdot(N-2)\cdot 2C\cdot\bigl((r+\ln2+c_3)^3-c_3^3\bigr)\text{.}
\end{equation}
This concludes our proof in the three dimensional case.
\end{proof}

Le us remark that if we link this to our study of the asymptotic volume of the Hilbert geometry of 
polytopes \cite{ver2012} we obtain the following corollary

\begin{cor}
  Let ${\mathcal P}_N$ be an open convex polytope  with $N$ vertices in $\R^n$, for $n=2$ or $3$, 
then there are three constants $\alpha_n$, $\beta_n$ and $\gamma_n$ such that for any point $p\in {\mathcal P}_N$ one has
$$
\alpha_n \cdot N\leq \liminf_{R\to +\infty}  \frac{\vol_{\mathcal{P}_N}B_{{\mathcal P}_N}(p,R)}{R^n} \leq \beta_n\cdot N+\gamma_n
$$
\end{cor}

Now let us come back to our initial problem and see how theorem 
\ref{growthpolytope} implies theorem \ref{pdmajorant}.

\begin{proof}[Proof of theorem \ref{pdmajorant} ]
We remind the reader that $\vol_{n-1,\Omega}$ stands for the $n-1$-dimensional 
Holmes-Thompson measure.
Let  $o$  be the centre of the Lowner ellipsoid of $\Omega$ which is supposed to be the unit Euclidean ball. 
We consider $R$ large enough in order to have the Euclidean
ball of radius $1/2n$ inside all the asymptotic balls involved in the sequel.

The idea of the proof consists in replacing for all $R$ large enough 
the convex set $\Omega$ by a  convex polytope $\mathcal{P}_R$ such that
\begin{itemize}
\item $\mathcal{P}_R$ is a subset of $\Omega$;
\item The asymptotic ball $P_R$ of the polytope $\mathcal{P}_R$ is inside the
$\bigl(1-\tanh(R)\bigr)/2n$-Euclidean neighborhood
of the corresponding asymptotic ball $\asb_\Omega(o,R)$ of $\Omega$.
\item  the exponential volume growth, with respect
to the geometry of $\Omega$, of the two
families of asymptotic balls  $(P_R)_{R\in\R}$
and $(\asb_\Omega(o,R))_{R\in \R}$ are the same. 
\end{itemize}
Let us insist on the fact that the convex polytope $\mathcal{P}_R$
depends on $R$.

Then using Theorem~\ref{growthpolytope} we will bound from above the area in 
dimension three or the perimeter in dimension two of the convex polytope $P_R$ by
a function depending linearly
on the number of verticies of $P_R$ and polynomialy on $R$. 
This will allow us to conclude.

\begin{figure}[hbtp]
  \centering
  \includegraphics{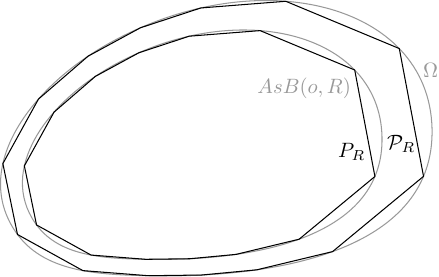}
  \caption{The asymptotic ball and an approximating polytope}
\end{figure}

Fix $R$. Among all
polytopes included both in the asymptotic ball $\asb_\Omega(o,R)$ and its
 $\bigl(1-\tanh(R)\bigr)/2n$ Euclidean-Hausdorff neighborhood pick
a polytope $P_R$ with the minimal number of verticies $N(R)$. Notice that we have 
\begin{equation}
  \label{eqNR}
  N(R)=N\biggl(\frac{1-\tanh(R)}{2n\tanh(R)}, \Omega\biggr)\text{.}
\end{equation}

\noindent\textbf{Claim:} There exists a constant $C>0$ such that for all $R$ the following inclusions
occur
\begin{equation}
  \label{eqinclusionpr}
   \asb_\Omega\bigl(o,R-C\bigr)\subset P_R\subset\asb_\Omega(o,R)\text{.}
\end{equation}

To prove this claim, on the one hand
we deduce from the first inclusion of Lemma~\ref{keyrmk} that 
$$
B_\Omega(o,R-\ln 2)\subset \asb_\Omega(o,R)\text{.}
$$
On the other hand the comparison of both Hilbert and Euclidean-Hausdorff neighborhoods, as stated in
Proposition~\ref{keyprop}, implies that the convex polytope $P_R$ lies in the 
$(\ln 3)/2$-Hilbert-Hausdorff neighborhood of the asymptotic ball  $\asb_\Omega(o,R)$. From these we deduce
the inclusion
\begin{equation}
  \label{eqincl1}
  B_\Omega\bigl(o,R-\ln 6\bigr)\subset P_R\subset\asb_\Omega(o,R)\text{.}
\end{equation}
Taking into account the second inclusion of Lemma~\ref{keyrmk} we finaly get
\begin{equation}
  \label{eqincl1}
  \asb_\Omega\bigl(o,R-\ln 6-\ln(n+1)\bigr)\subset P_R\subset\asb_\Omega(o,R)\text{,}
\end{equation}
which proves our claim with $C=\ln6+\ln(n+1)$.

Thanks to the  monotonicity of the Holmes-Thompson measure (see Lemma~\ref{l:htmonotonicity})  
we know that the area of the boundary $\partial P_R$ is less than the area
of the asymptotic sphere $AsS_\Omega(o,R)$, 
but larger than the area of the asymptotic
sphere of radius $R-C$, that is:
\begin{equation}
  \label{eq:upperboundzero}
  \vol_{n-1,\Omega}\bigl(AsS_\Omega(o,R-C)\bigr)\leq \vol_{n-1,\Omega}(\partial P_R)\leq \vol_{n-1,\Omega}\bigl(AsS_\Omega(o,R)\bigr)\text{.}
\end{equation}

From the equation~(\ref{eq:upperboundzero}) we deduce that
the logarithms of the areas of $\partial P_R$ and $AsS_\Omega(o,R)$ are asymptotically 
the same in the following sense
\begin{equation}
  \label{eq:upperboundequivalence}
  \lim_{R\to +\infty} \dfrac{\ln \vol_{n-1,\Omega}\bigl(AsS_\Omega(o,R)\bigr) }{\ln \vol_{n-1,\Omega}(\partial P_R) }=1 \text{.}
\end{equation}

Let us denote by ${\mathcal P}_R$ the image of $P_R$ by the dilation of ratio $1/(\tanh R)$. This is the dilation sending $\asb_\Omega(o,R)$ to $\Omega$. Hence,  by construction, 
${\mathcal P}_R \subset \Omega$ and therefore we have
\begin{equation}
  \label{eq:upperboundone}
  \vol_{n-1,\Omega}(\partial P_R)\leq \vol_{n-1,{\mathcal P}_R}(\partial P_R)\text{.}
\end{equation}

Now thanks to Theorem~\ref{growthpolytope}, for $n=2$ or $n=3$ and $R>0$ 
such that $\tanh(R)>3/4$, 
there are  two constants $a_n$, $b_n$ and a polynomial 
$Q_n$ of degree $n$ such  that
\begin{equation}
  \label{eq:upperboundtwo}
  \vol_{n-1,\Omega}(\partial P_R)\leq (a_nN(R)+b_n) Q_n(R)\text{.}
\end{equation}

To conclude remark that
$$
\liminf_{R\to +\infty} \frac{\ln(N(R))}{R} = 2\underline{a}(\Omega)\text{, and } \limsup_{R\to +\infty} \frac{\ln(N(R))}{R} = 2\overline{a}(\Omega)\text{,}
$$
and use it with the inequality~(\ref{eq:upperboundtwo}) to get for instance
$$
\limsup_{R\to +\infty} \dfrac{\ln \vol_{n-1,\Omega}(\partial P_R)}{R} \leq 2\overline{a}(\Omega). 
$$
Finally the limit~(\ref{eq:upperboundequivalence}) implies that
$$
 \limsup_{R\to +\infty} \dfrac{\ln \vol_{n-1,\Omega}\bigl(AsS_\Omega(o,R)\bigr) }{R} \leq 2\overline{a}(\Omega).
$$
The left hand sides of this last inequality is easily seen to be the spherical
entropy (see (\ref{sphericalentropy})), which ends our proof.
\end{proof}

The following corollary follows from Bron\u ste\u\i n and Ivanov's Theorem~\ref{bitheo}
which states that $2\overline{a}\leq n-1$.

\begin{cor}
  Let $\Omega$ be an open bounded convex set in $\R^n$, for $n=2$ or $3$, then
$$
\Ent(\Omega)\leq n-1\text{.}
$$
\end{cor}

We are now going to study the reverse inequality.

\begin{theo}\label{pdminoreent}
  Let $\Omega$ be an bounded convex domain  in $\R^n$.
The volume entropies of $\Omega$ are bigger or equal to twice the approximabilities of $\Omega$, i.e.,
$$
2\underline{a}(\Omega)\leq \lent(\Omega)\text{ and }2\overline{a}(\Omega)\leq \Ent(\Omega)\text{.}
$$
\end{theo}

\begin{proof}[Proof of Theorem~\ref{pdminoreent}]
Without loss of generality we suppose that the Euclidean unit ball is 
the Lowner ellipsoid of $\Omega$, 
and  $o$ is the centre of that ball.

The idea of the  proof is the following:
\begin{itemize}
\item We will show that for a good positive $\delta$ and any positive real 
number  $R$ there exists a $\delta$-separated set $\mathcal{S}_R$ in the metric ball of radius $B(o,R+2\delta)$,
such that the convex closure $P_R$ of that set contains the ball $B(o,R)$.
\item We will then use the fact that the cardinal of this $\delta$-separated set will be larger than the cardinal of the set of verticies of a vertex minimising
convex polytope included in the annulus $B(o,R+2\delta)\setminus B(o,R)$.

In other words, the number of points in the $\delta$-separated 
will be bounded from below by the number 
$N\bigl(\varepsilon(R),\Omega\bigr)$ from the introduction. Here  
$\varepsilon$ will be a function of $R$.
\item To conclude 
we will take into account that the union of the open metric balls of radius $\delta/2$ centred  at the point of the $\delta$-separated set $\mathcal{S}_R$ are disjoints and are in the ball $B(o,R+3\delta)$. Thus getting a lower bound on the volume of the ball $B(o,R+3\delta)$ in terms of $N\bigl(\varepsilon(R),\Omega\bigr)$ times a constant depending on the dimension.
\end{itemize}

Let us now start the proof. Consider the $(\ln 3)/2$-Hilbert neighborhood of the metric ball $B(o,R)$, that is 
$$V(R)=B(o,R+(\ln 3)/2)\text{,}$$  
and take a maximal $\delta=(\ln 3)/4$-separated set $\mathcal{S}_R$ on its boundary. 
This set contains $\#\mathcal{S}_R$ points.

Now let us take the convex hull $\mathcal{C}_R$ of these points. 
This is a polytope  with $N_2(R)\leq \#\mathcal{S}_R$ vertices.

\begin{claim}\label{keyclaimlb}
The polytope $\mathcal{C}_R$ is
included in the $2\delta$-Hilbert neighborhood
of $B(o,R)$ and contains $B(o,R)$. 
\end{claim}

Notice that if the claim holds, then for some
real constant $c$ independent of $R$ (see corollary \ref{keycor} once again), we have
\begin{equation}
  \label{eqlbent}
  \#\mathcal{S}_R\geq N_2(R)\geq \widetilde N(R-c):=N\bigl(1-\tanh(R-c)/4,\asb(o,R-c)\bigr).
\end{equation}

\begin{proof}[Proof of claim~\ref{keyclaimlb}]
First notice that $V(R)$ is a convex set (see Busemann \cite{busemann}, chapter II, section 18, page 105).
Therefore the convex hull is inside the $2\delta$-Hilbert neighborhood of $B(o,R)$,
that is $V(R)$.

Now let us suppose by contradiction that $\mathcal{C}_R$ does not contain $B(o,R)$.
Hence there exists some points $q$ in $B(o,R)$ which is not in $\mathcal{C}_R$.

We will show that we can find a point on the sphere $S(o,R+2\delta)$ which
is at a distance bigger that $\delta$ from all points of $\mathcal{S}_R$, which
will contradict its maximality.

Under our assumption, the  Hahn-Banach separation theorem asserts that there
 exists a linear form $a$, 
some constant $c$ and a hyperplane $H=\{x\mid a(x)=c\}$ 
which separates $q$ and $\mathcal{C}_R$, i.e., 
$a(q)>c$ and $a(x)<c$ for all $x\in \mathcal{C}_R$. 
Consider then $H_q=\{x\mid a(x)=a(q)\} $ the parallel hyperplane to 
$H$ containing $q$. Let us say that a point $x$ such that $a(x)\geq a(q)$ is \textsl{above} the hyperplane $H_q$.

Then let us define by $V_o'=\{ x\in \partial V(R)\mid a(x)\geq a(q)\}$ the part of the boundary of 
$V(R)$ which is above $H_q$.

Now we want to  metrically project each point of $V_o'$ onto $H_q$, 
that is to say that
to each point of $V_o'$ we associate its closest point on $H_q$.

However if $\Omega$ is not strictly convex, the projection might not be unique (see the appendix \ref{Aprojection}), 
that is why we are going to distinguish two cases.

\noindent\textsl{First case:} 
The convex set $\Omega$ is strictly convex, then the metric projection
is a map from $V_o'$ to $H_q$ and it is continuous, furthermore the point
on $H_q\cap V_o'$ are fixed and by convexity $H_q\cap V_o'$ 
is homeomorphic to a $n-2$-dimensional sphere. 
Therefore by Borsuk-Ulam's theorem (or its version known as the \textsl{antipodal map theorem}), there is
a point $p$ on $V_o'$ whose metric projection is $q$.

Now as $p$ is on the boundary of $V(R)$, that is the sphere $B(o,R+2\delta)$,
and $q$ is in $B(o,R)$ we  necessarily have 
$$d_\Omega(p,q)\geq (\ln 3)/2\text{.}$$
hence for all points $x$ in $H_q\cap V_o'$, we have 
$$d_\Omega(p,x)\geq d_\Omega(p,q)\geq (\ln 3)/2\text{.}$$

\noindent\textsl{Second case:} 
The convex set $\Omega$ is not strictly convex. Then let us approximate
it by a smooth and strictly convex set $\Omega'$ such that $\Omega\subset \Omega'$, and for all
pair of points $x, y \in V(R)$, 
\begin{equation}\label{smoothapprox}
\frac23\times d_{\Omega'}(x,y)\geq d_{\Omega}(x,y)\geq d_{\Omega'}(x,y).
\end{equation}

Then metrically project $V_o'$ onto $H_q$ with respect to $\Omega'$. By the same argument
as in the first case, we obtain a point $p$ such that for all $x$ in $H_q\cap V_o'$ we
have
$$
d_{\Omega'}(p,x)\geq d_{\Omega'}(p,q)\geq \frac32 d_\Omega(p,q)\geq \frac34 (\ln 3)
$$

which also implies by the inequalities~(\ref{smoothapprox})
 that for all $x$ in $H_q\cap V_o'$ we
have
$$
d_\Omega(p,x)\geq 3(\ln 3)/4.
$$

In either cases, using the Lemma~\ref{spheremonotonic} of 
the section~\ref{Adistanboule},
we deduce that all points on $\partial V_R$ at distance less or equal to $(\ln 3)/4$
from $p$ are above $H_q$ and are therefore 
contained in $V_o'$.  We then infer that there are no points of $\mathcal{S}_R$
at distance less or equal to $(\ln 3)/4$ from $p$, 
which contradicts the maximality of the set $\mathcal{S}_R$. 
\end{proof}

Now consider the union of the balls of radius $\delta/2$ centred at the points
of $\mathcal{S}_R$. This union is a subset of the ball $B(o,R+3\delta)$
and the balls are mutually disjoint. Now following our paper~\cite{ver2012},
there exists a constant $a_n$ such that for any open proper convex $\Omega$ and $x\in \Omega$,
the volume of the ball of radius $r$ centred at $x$ is at least $a_nr^n$.
Hence from this fact and the inequality~(\ref{eqlbent}) we get that
for all $R>0$,
\begin{multline}
  \vol_\Omega\bigl(B(o,R+3\delta)\bigr)\geq \#\mathcal{S}_R \cdot a_n\delta^n\\\geq N(1-\tanh(R-c)/4,\asb(o,R-c))\cdot a_n\delta^n\text{.}
\end{multline}

Now if we take the logarithm of the previous inequalities, divide by $R$ and
take either the $\liminf$ or the $\limsup$ we conclude the proof of the Theorem \ref{pdminoreent}.
\end{proof}

The proof of the main Theorem~\ref{MainTheo}  is now complete, and we now turn to its corollaries.

A point $x$ of a convex body $K$ is called a \textsl{farthest point} of $K$ if and only if, for some
point $y\in \R^n$, $x$ is farthest from $y$ among the points of $K$. The set of farthest points of $K$,
which are special exposed points, will be denoted by $exp^*K$. Thus a point $x\in K$
belongs to $exp^*K$ if and only if there exists a ball which circumscribes  $K$ and contains $x$ in its
boundary.

In dimension $2$ we get the following corollary,
\begin{cor}
  Let $\Omega$ be a plane Hilbert geometry, and let $d_M$ be the Minkowski
dimension of extremal points and $d_H$ the  Hausdorff dimension
of the set $exp^*\Omega$ of farthest points then we have the following inequalities
\begin{equation}
  \label{eqentfarthestpoints}
  d_H\leq \lent(\Omega) \leq \Ent(\Omega)\leq \frac{2}{3-d_M}\text{.}
\end{equation}
The left hand side inequality remains valid for higher dimensional Hilbert geometries.
\end{cor}
\begin{proof}
  The left hand side of inequality (\ref{eqentfarthestpoints}) comes from R.~Schneider and J.~A.~Wieacker~\cite{sw},
whereas the right hand  one is the First main Theorem in G.~Berck, A.~Bernig and C.~Vernicos~\cite{berck_Bernig_Vernicos}.
\end{proof}

\begin{rmq}
  Inequality (\ref{eqentfarthestpoints}) induces a new result concerning the approximability in dimension $2$,
as it implies that
$$
\overline{a}(\Omega)\leq \frac{1}{3-d}.
$$
\end{rmq}

Lastly we are also able to prove the following result which relates the entropy of a convex set and the
entropy of its polar body.

\begin{cor}\label{cvxeAndPolarSameEntropy}
Let  $\Omega$ be a Hilbert geometry of dimension $2$ or $3$, then
$$
\lent(\Omega)=\lent(\Omega^*)\text{, and }\ \Ent(\Omega)=\Ent(\Omega^*)
$$
\end{cor}
\begin{proof}

  It suffices to prove that the approximability of a convex body $\Omega$ containing the origin and its polar $\Omega^*$ are equal. 
Without loss of generality we can assum that the unit 
ball is $\Omega$'s John's ellipsoid. Hence $\Omega$ is contained in the ball of radius 
the dimension and its polar contains the ball of radius the inverse of the dimension and is included in the unit ball.
Now, notice that for $\varepsilon$ small enough, if $P_k$ is a polytope with $k$ vertexes inside the $\varepsilon$-Hausdorff neighborhood
of $\Omega$, then its polar $P_k^*$ is a polytope with $k$ faces containing $\Omega^*$ and contained in its $\varepsilon\cdot C$-Hausdorff neighborhood, for some constant $C$ depending only on the dimension.
A known fact (see Gruber~\cite{gruber} section 11.2) states that the approximability can be computed either by minimising the
vertexes or the faces. 
Hence $\underline{a}(\Omega)=\underline{a}(\Omega^*)$ and $\overline{a}(\Omega)=\overline{a}(\Omega^*)$.
The statement therefore follows from the Main Theorem. 
\end{proof}

\section{Intermediate growth}\label{growth}

\begin{quote}
  In this section we focus on the two dimensional case.
\end{quote}

The intermediate volume growth will follow from Theorem~\ref{growthpolytope} and 
the following Proposition,
which allows us to control both the length of sphere and their volume in dimension $2$
from below, thanks to the number of verticies of an ad-hoc approximating
polytope, in the fashion of Theorem~\ref{growthpolytope}, except that here
the lower bounds depend  on $\Omega$.

\begin{prop}\label{lbounddeuxd}
Let $\Omega$ be an open bounded convex set in $\R^2$ whose Lowner ellipsoid is the Euclidean unit ball centered at $o\in \Omega$. Let $N(\epsilon,\Omega)$ be the minimal number of verticies 
of a polygon containing $\Omega$ at Euclidean-Hausdorff distance less that $\epsilon$ from $\Omega$, and to any positive real number $R$ let $N(R):= N(\frac{1-\tanh(R)}{4\tanh(R)},\Omega)$.

Then there exists three constants $R_2$, $K_2$ and $C_2$ independant of $\Omega$, such that
for all real numbers $R>R_2$ we have
\begin{equation}\label{minorationdeuxd}
  \begin{array}{rcl}
     \text{Length}_\Omega\bigl(S_\Omega(o,R)\bigr)&\geq& \bigl(N\bigl(R-(3/2)\ln 3\bigr)-2\bigr) K_2\text{,}\\
     \vol_{\Omega} \bigl(B_\Omega(o,R+K_2/2)\bigr)&\geq&\bigl(N\bigl(R-(3/2)\ln 3\bigr)-2\bigr) C_2 (K_2)^2\text{.} 
  \end{array}
\end{equation}
The same result holds for the asymptotic balls with $R>R_2+\ln2$.
\end{prop}

We want to stress out once again that there is actually no
loss in generality in supposing the Euclidean unit ball to be the Lowner ellipsoid
of $\Omega$.

\begin{proof}
For any real positive number $R$ let $\epsilon(R)=\bigl(1-\tanh(R)\bigr)/4$.

The idea is to built a convex polygone in the $\epsilon(R)$-neighborhood of an asymptotic ball
of radius $R$ in a way we can control uniformly 
from below the length of the edges. 

More precisely we have the following.
\begin{claim}
  There exist a convex polygone $\mathcal{P}_R$ such that
\begin{itemize}
\item $\mathcal{P}_R$ contains the asymptotic ball $\asb(o,R)$ and  
is in its $\epsilon(R)$ Haus\-dorff-Euclidean
neighborhood;
\item  All the edges of $\mathcal{P}_R$ but one are  tangent
to $\asb(o,R)$ and  all its vertexes belong to the boundary $\partial_R\asb$ of the  
$\epsilon(R)$-Hausdorff neighborhood of the asymptotic ball  $\asb_\Omega(o,R)$.
\end{itemize}
\end{claim}

This claim is a consequence of the following algorithm: 
\begin{enumerate}[\textsl{Step} 1]
\item Draw one tangent to $\asb_\Omega(o,R)$, it will meet the boundary $\partial_R\asb$ of its $\epsilon(R)$-Hausdorff neighborhood at two  points $x_1$ 
and  $x_2$, where $\vect{ox_1}$, $\vect{ox_2}$ are positively oriented.
\item We start from
$x_2$ and draw the second tangent to $\asb_\Omega(0,R)$ passing by $x_2$. This second tangent will meet the boundary $\partial_R\asb$ at a second point $x_3$.
\item for $k>2$, if the second tangent $t_{k+1}$ to $\asb_\Omega(0,R)$ passing by $x_k$ has its second intersection with $\partial_R\asb$
on the arc of from $x_1$ to $x_k$ (in the orientation of the construction), we stop and consider
for $\mathcal{P}_R$ the convex hull of $x_1,\ldots,x_k$, otherwise we take for $x_{k+1}$ that  second intersection  
of the tangent $t_{k+1}$ with $\partial_R\asb$ and start again that step.
\end{enumerate} 

This algorithm will necessarily finish, because by convexity 
the arclength of $x_ix_{i+1}$ on $\partial_R\asb$ built this way
is bigger than $2\epsilon(R)$. 
At the end of this algorithm we obtain, by minimality, a polygon which has
at least $N(R)=N\bigl(\epsilon(R),\asb_\Omega(o,R)\bigr)=N\bigl(\epsilon(R)/\tanh(R),\Omega\bigr)$ edges.

Recall that Proposition~\ref{keyprop} guaranties us
that the $\epsilon(R)$-Euclidean neighborhood of the asymptotic ball $\asb_\Omega(o,R)$ 
is included in its $(\ln 3)/2$-Hausdorff-Hilbert neighborhood and therefore,
taking into account the inclusions (\ref{inclusions}), we obtain
$$ 
B_\Omega(o,R-\ln 2)\subset \asb_\Omega(o,R)\subset\mathcal{P}_R\subset B_\Omega\bigl(o,R+(3/2)\ln 3\bigr)\text{.}
$$  

Moreover, the length coincides with the Holmes-Thompson $1$-dimensional measure. 
Therefore, the  monotonicity of the later, 
as seen in Lemma \ref{l:htmonotonicity}, implies the following inequalities:

\begin{equation}\label{ineqlongueur}
  \begin{split}
     \text{length}_\Omega S_\Omega(o,R-\ln 2)\leq& \text{length}_\Omega \partial\asb_\Omega(o,R) \\
&\leq \text{length}_\Omega \partial\mathcal{P}_R \leq \text{length}_\Omega S_\Omega(o,R+3/2\ln 3)\text{.}
  \end{split}
\end{equation}

Now let $\mathfrak{P}_R$ be the image of $\mathcal{P}_R$ under the dilation of ratio $\tanh(R)^{-1}$ centred
at $o$. By construction $\mathfrak{P}_R$ contains $\Omega$,  which implies
$$
\text{Length}_{\mathfrak{P}_R} \partial\mathcal{P}_R \leq \text{length}_\Omega \partial\mathcal{P}_R\text{.}
$$

Therefore it sufficies to prove the following claim:

\begin{claim}\label{kclaim}
  Let $I(R)\in \partial_R\asb$ be a vertex of $\mathcal{P}_R$, such that 
the two edges containing $I(R)$ are
tangent to $\asb_\Omega(o,R)$ at $b(R)$ and
$c(R)$. Then for any $R> \tanh^{-1}(1/2)=R_2$
$$
d_\Omega(b(R),c(R))\geq d_{\mathfrak{P}_R}\bigl(b(R),c(R)\bigr) \geq \ln(6/5)=K_2\text{.}
$$
\end{claim}

Indeed, let us assume that claim~\ref{kclaim} is true, and for $R>r_2$
consider a vertex $v$ of $\mathcal{P}_R$ whose incident edges are tangent 
to $\asb(o,R)$.
Let $b$ and $c$ the two points of tangency, then by the triangle inequality, 
$$
d_\Omega(b,v)+d_\Omega(c,v)\geq d_\Omega(b,c)\geq K_2\text{.}
$$
Therefore the length of $\mathcal{P}_R$ is bigger than $\bigl(\widetilde N(R)-2\bigr)K_2$, where $\widetilde N(R)$
is number of edges of $\mathcal{P}_R$
(because of the possible exception at $x_1$ and the last point of the construction above). 
Hence taking $R_2=r_2+(3/2)\ln3$, thanks to the equation~(\ref{ineqlongueur}), we get
for $R>R_2$
\begin{equation}
  \label{eqlowerlength}
  \text{Length}_\Omega\bigl(S_\Omega(o,R)\bigr)\geq \Bigl(\widetilde N\bigl(R-(3/2)\ln 3\bigr)-2\Bigr) K_2\text{,}
\end{equation}
and as $\widetilde N\bigl(R-(3/2)\ln 3\bigr) \geq N\bigl(R-(3/2)\ln 3\bigr)$ the first inequality in~(\ref{minorationdeuxd}) is proved.

Now concerning the volume of the ball, 
Claim \ref{kclaim} and Proposition \ref{spheremonotonic} imply
that the contact points of the edges of $\mathcal{P}_R$ with $\asb_\Omega(o,R)$ 
form a $K_2$ separated set. Hence we can conclude in the same
way as we did during  the Proof of Theorem~\ref{pdminoreent}, \textsl{i.e.},
the balls of radius $K_2/2$ centred at those points are disjoint and 
included in the metric ball  
$B_\Omega\bigl(o,R+(3/2)\ln 3+(K_2/2)\bigr)$. 
Now following \cite{ver2012}, there exists a constant $C$ depending
only on the dimension such that the volume of the ball of radius $r$ is at 
least $C\cdot r^2$. Hence we obtain
that
\begin{equation}
\vol_\Omega B_\Omega\bigl(o,R+(3/2)\ln 3+(K_2/2)\bigr) \geq (\widetilde N(R)-2)\cdot C \cdot(K_2/2)^2 \text{,}
\end{equation}
and the last inequality~(\ref{minorationdeuxd}) follows onece again from the inequality 
$\widetilde N(R)\geq N(R)$.

\bigskip

\textsl{Proof of the Claim~\ref{kclaim}.} 

Let $a(R)$ (resp. $d(R)$) be the opposite vertex to $I(R)$ on the edge containing $b(R)$ (resp. $c(R)$).

Now let us consider the images $I$, $a$,$b$, $c$ and $d$ of the five points 
$I(R)$, $a(R)$, $b(R)$, $c(R)$ and $d(R)$ by the dilation of ratio
$1/\tanh R$ centred at $o$. Then we are in  the same configuration
as in the claim \ref{lessegments}, with $\mathfrak{P}_R$ instead of $\Omega$. Let 
 $u(R)=\frac{bc}{BC}\frac{\tanh(R)}{1-\tanh(R)}$, then
following (\ref{minorationlg}) we have
$$
d_{\mathfrak{P}_R}\bigl(b(R),c(R)\bigr)\geq \frac12 \ln\Bigl(1+\frac{u(R)+u(R)^2}{s(1-s)}\Bigr) \text{.}
$$

Therefore we need to obtain a lower bound for $u(R)$. To do this, let
$p$ be the intersection of the line $oI$ with the lines $(bc)$. 
Then thanks to Thales's theorem we have
$$
\frac{BC}{bc}=\frac{oI}{pI}=\frac{op+pI}{pI}=1+\frac{op}{pI}
$$

Concerning the distance $op$, recall that the unit ball centred at $o$ is the Lowner ellipsoid  of $\Omega$ and therefore we get $op\leq\frac{1}{\tanh(R)}$,
 because by convexity $p$ is in $\Omega$.

Regarding the distance $pI$, as $I(R)$ is on the boundary of the  $(1-\tanh(R))/4$ Euclidean neighborhood of $\asb(o,R)$, 
$I$ is on  the boundary of the  $\bigl(1-\tanh(R)\bigr)/4\tanh(R)$ neighborhood of $\Omega$. 
Hence we obtain 
$$pI\geq \bigl(1-\tanh(R)\bigr)/4\tanh(R)\text{,}$$
 because the segment $[p,I]$ intersects $\Omega$. 
This way we obtain
$$
\frac{BC}{bc}\leq 1+\frac{4}{1-\tanh(R)}
$$
which in turn implies that
$$
1 \leq \frac{5-\tanh(R)}{1-\tanh(R)}\frac{bc}{BC}\leq \frac{5}{1-\tanh(R)}\frac{bc}{BC}\text{.}
$$
Hence
\begin{equation}
  \label{eqminorationlg}
  \frac{\tanh(R)}{5} \leq u(R)
\end{equation}
Therefore if $\tanh(R_2)=1/2$ then for all $R>R_2$ we get $10u(R)>1$.

Finally using the fact that $s(1-s)\leq1/4$ and taking $R>R_2$ we get
$$
d_{\mathfrak{P}_R}\bigl(b(R),c(R)\bigr)\geq\frac12 \ln\biggl(1+\frac{2}{5}+\frac{1}{25}\biggr)=\ln(6/5)>0.18\text{.}
$$
\end{proof}

\begin{proof}[Proof of the intermediate volume growth theorem]

  Following Schneider and Wieacker \cite[theorem 4, p. 154]{sw} and its proof, 
for any increasing function  $f\colon\R^+\to \R^+$  such that
$$
\liminf_{r\to +\infty} \frac{e^r}{f(r)}>0
$$
there exists a convex set $\Omega_f$ such that 
\begin{equation}
  \label{eq:inegalitedouble}
  0< \liminf_{r\to +\infty} \frac{N\bigl(1-\tanh(r),\Omega_f\bigr)}{f(r)}\leq \limsup_{r\to +\infty} \frac{N\bigl(1-\tanh(r),\Omega_f\bigr)}{f(r)} < +\infty\text{.}
\end{equation}

In the sequel we will denote $N(r)=N\bigl(1-\tanh(r),\Omega_f\bigr)$ and 
drop the the subscript $\Omega_f$ in the notation of metric and asymptotic balls.

Now let $o$ be the center of the lowner ellipsoid of $\Omega_f$.
Following Proposition \ref{lbounddeuxd} for $K_2=\ln(6/5)$ and $r>0$ satisfying  
$$\tanh\bigl(r-(3/2)\ln 3-K_2/2\bigr)\geq 1/2$$ we have that
\begin{equation}
  \label{eq:application}
  \begin{array}[t]{rcl}
    \vol_{\Omega_f}\bigl(B(o,r)\bigr)&\geq& \Bigl(N\bigl(r-\frac32\ln 3-K_2/2\bigr)-2\Bigr) C (K_2)^2\text{.}
  \end{array}
\end{equation}
This inequality  implies that
\begin{equation}\label{eq:igone}
  \liminf_{r\to +\infty} \frac{\vol_{\Omega_f}\bigl(B(o,r)\bigr)}{f(r)}\geq C (K_2)^2\liminf_{r\to +\infty} \frac{N\bigl(r-\frac32\ln 3-K_2/2\bigr)-2}{f(r)}\text{.}
\end{equation}

Now using inequalities (\ref{eq:upperboundzero}) to (\ref{eq:upperboundtwo}) from Theorem  \ref{pdmajorant} proof's 
we get the existence of three
constants $a$, $b$ and $c$ such that if $K=\ln 18$ and $r>0$ is a real number satisfying $\tanh(r-C)>3/4$ then
\begin{equation}\label{eq:applicationbis}
  \vol_{\Omega_f}\bigl(\asb(o,r-C)\bigr)\leq N\biggl(\frac{1-\tanh(r)}{4\tanh(r)},\Omega_f\biggr) (ar^2+br+c)\text{.}
\end{equation}

The inclusion $B(o,r-\ln(2)-C)\subset \asb(o,r-C)$ given by  (\ref{inclusions}) in Lemma \ref{keycor} proof's 
allow us to obtain the next inequality:
\begin{equation}
  \vol_{\Omega_f}\bigl(B(o,r-C-\ln 2)\bigr)\leq N\biggl(\frac{1-\tanh(r)}{4\tanh(r)},\Omega_f\biggr) (ar^2+br+c)\text{,}
\end{equation}
which in turm implies that
\begin{equation}\label{eq:igtwo}
  \limsup_{r\to +\infty} \frac{\vol_{\Omega_f}\bigl(B(o,r)\bigr)}{r^2f(r)}\leq a \times \limsup_{r\to +\infty} \frac{N\bigl(\frac{1-\tanh(r)}{4\tanh(r)},\Omega_f\bigr)}{f(r)} \text{.}
\end{equation}

Combining both inequalities (\ref{eq:application}) and \ref{eq:applicationbis}) and using
the asymptotic comparison( \ref{eq:inegalitedouble}) we finally conclude that
$$
\liminf_{r\to +\infty} \frac{\ln \vol_{\Omega_f}\bigl(B(o,r)\bigr)}{r} = \liminf_{r\to +\infty} \frac{\ln f(r)}{r}\text{.}
$$

In the above proofs we can replace $\liminf$ by $\limsup$.

To obtain the penultimate statement consider $f(r)=e^r/r^3$, and apply our result to get a convex set $\Omega_f$
whose entropy is $1$.
However, by definition of the centro-projective area and our result in the two dimensional case \cite{berck_Bernig_Vernicos} we have
\begin{multline}
 \mathcal{A}_o(\Omega_f) =\lim \frac{\vol_{\Omega_f}\bigl(B(o,r)\bigr)}{\sinh r}=\limsup \frac{\vol_{\Omega_f}\bigl(B(o,r)\bigr)}{\sinh r}\\
= \limsup \frac{\vol_{\Omega_f}\bigl(B(o,r)\bigr)}{e^rr^{-1}}\times \frac{e^r}{r\sinh r}=0.
\end{multline}

For the last statement take $f(r)=r^3$ and apply our result to get a convex $\Omega_f$ such that 
$$
\limsup{\frac{\vol_{\Omega_f}\bigl(B(o,r)\bigr)}{r^2}}=\limsup{\frac{r\vol_{\Omega_f}\bigl(B(o,r)\bigr)}{r^3}}=+\infty
$$
hence following our paper \cite{ver2012}, $\Omega_f$ is not a polytope. Furthermore the entropy of such a $\Omega_f$ is zero as we have
$
\limsup_{\infty} \ln(r^3)/r=0.
$

\end{proof} 

To conclude this section let us show how Corollary \ref{cor:volentropyone} related to the values attained by the lower and upper volume entropies easily follows: 
Suppose first that $0<\alpha\leq \beta\leq 1$,
and start by considering a sequence  $(U_n)_{n\in\N}$ defined for some $x>0$ by  $U_0=e^{bx}$, and for all $k\geq0$ by $$U_{2k+1}=e^{\alpha U_{2k}} \text{, and } 
U_{2k+2}=e^{\beta U_{2k+1}}\text{.}$$
Then take an increasing function 
$f\colon\R^+\to \R^+$ such that for all
$r\in \R$,
$$
e^{\alpha r}\leq f(r)\leq e^{\beta r}\text{,}
$$
and $f(U_n)=U_{n+1}$ for all $n\geq 0$. We can define such a function piecewise linearly.

If $\alpha=0$, replace $r\mapsto e^{\alpha r}$ by $r\mapsto 2r$ above and take $U_{2k+1}=2U_{2k}$ for all $k\geq 0$.





\appendix

\section{Metric projection \\ in a Hilbert geometry}\label{Aprojection}
The following is a reformulation and a detailed proof of a statement found in section 21 and 28 of Busemann-Kelly's book \cite{busemannkelly}
in any dimension.

\begin{prop}
  Let $(\Omega,d_{\Omega})$ be a Hilbert geometry in $\R^n$. Let $p$ be a point of $\Omega$ and $H$ an hyperplane intersecting $\Omega$.
Then $q\in H\cap\Omega$ is a \textsl{metric projection} of $p$ onto $H$, i.e.,
$$
d_\Omega(p,H)=d_\Omega(p,q)\text{,}
$$ 
if and only if $\partial\Omega$ has, at its intersection with the straight line $(pq)$, supporting hyperplanes concurrent
with $H$ (the intersection of these three hyperplanes is an $n-2$-dimensional affine space).
\end{prop}
\begin{proof}
Let us suppose first that such concurrent support hyperplanes exists.
  Let $x$ and $y$ be the intersections of the line $(pq)$ with $\partial\Omega$. Assume that $\xi$ and $\eta$ are supporting hyperplanes
of $\partial\Omega$ respectively at $x$ and $y$ whose intersection with $H$ is the $n-2$-affine space $W$.
Let us show that for any $p'\in (pq)$ and any $q'\in H$ we have
\begin{equation}
  \label{eqmetproj}
  d_\Omega(p',q')\geq d_\Omega(p',q)\text{.}
\end{equation}
Let us suppose that $x$ is on the half line $[qp')$ and $y$ on the half line $[p'q)$ and denote by $x'$ and $y'$ the
intersection of $\partial \Omega$ with the half line $[q'p')$ and $[p'q')$ respectively. Then let $x_0$ be the intersection
of $\xi$ with the line $(p'q')$ and $y_0$ the intersection of $(p'q')$ with $\eta$. By Thales' theorem, the cross-ratio
of $[x_0,p',q',y_0]$ is equal to the cross ratio of $[x,p',q,y]$ and standard computation shows
that $[x_0,p',q',y_0]\leq [x',p',q',y']$, with equality if an only if $x_0=x'$ and $y'=y_0$. Hence
the inequality (\ref{eqmetproj}) holds, and if the convex set is strictly convex, this inequality is always strict, for $q'\neq q$.

\begin{figure}[h]
  \centering
  \includegraphics{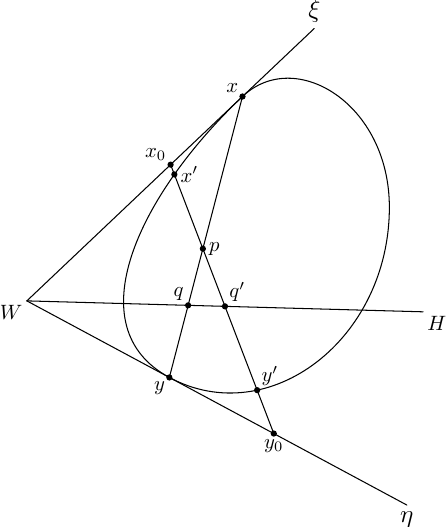}
  \caption{Metric projection of $p$ on $H$.}
\end{figure}
Reciprocally:
recall that when a point $q'$ of $\Omega$ goes to the boundary, its distance to $p$ goes to infinity. Hence by continuity of the distance and compactness
there exists a point $q$ on $H\cap\Omega$ such that $d_\Omega(p,H)=d_\Omega(p,q)$. Now consider the Hilbert ball $B_\Omega(p,r)$ of radius $r=d_\Omega(p,H)$ centred at $p$.
Let once more $x$, $y$, $\xi$ and $\eta$ be defined as before, and let $H'$ be the hyperplane passing by $q$ and $\xi\cap\eta=W$. Then this hyperplane
has to be tangent to  the ball $B_\Omega(p,r)$, otherwise  one can find  a point $q'$ on $H'$ inside the open ball (i.e. $d(p,q')<r)$, however 
by the reasoning done in our first step we would conclude that this point is at a distance bigger or equal to $r$, which would be a contradiction.
By minimality of the point $q$, $H$ is also a supporting hyperplane of $B_\Omega(p,r)$ at $q$. Hence we have to distinguish between two cases.
If $\Omega$ is $C^1$, then by uniqueness of the tangent hyperplanes at every point $H=H'$. Otherwise, $\Omega$ is not $C^1$ at $x$ or $y$.
In that case it is possible to change one of the hyperplane, say $\xi$, with $\xi'$ passing by $x$ and $H\cap\eta$ (which might be at infinity, which would
mean that we consider parallel hyperplanes).
\end{proof}

Notice that there is no uniqueness of the metric projections (also called "foot" by Busemann). However if the convex set is strictly convex,
then we will have a unique projection, if furthermore the convex is $C^1$, this projection will be given by a unique pair of
supporting hyperplanes.

\subsection{Approximability of convex bodies seen as a dimension}

\begin{quote}
In this appendix we relate our definition of approximability with the definition given in~\cite{sw}
\end{quote}

Recall that for a convex body  $\Omega$ and $\varepsilon >0$,  $N(\varepsilon,\Omega)$ denotes the smallest number
of vertices of a polytope whose Hausdorff distance
to $\Omega$ is less than $\varepsilon$.

The following result is due to R.~Schneider and J.~A.~Wieacker~\cite{sw}

\begin{theo}[\cite{sw}]\label{thmsw}
  Let $\underline{a}_s:=\liminf_{\varepsilon \to 0^+} N(\varepsilon,\Omega)\varepsilon^s$, then $s\to \underline{a}_s$ admits
a critical value $\underline{a}(\Omega)$ called \textsl{approximability number of} $\Omega$,  
such that, if $s>\underline{a}(\Omega)$ then
$\underline{a}_s(\Omega)=0$ and if $s< \underline{a}(\Omega)$ then $\underline{a}_s(\Omega)=\infty$.
\end{theo}

In the same way, we can introduce the 
\textsl{upper approximability number of} $\Omega$ , denoted by $\overline{a}(\Omega)$,
as the critical value of $s\mapsto \overline{a}_s(\Omega)$, where 
$$
\overline{a}_s(\Omega):=\limsup_{\varepsilon \to 0^+} N(\varepsilon,\Omega)\varepsilon^s\text{.}
$$

The reader familiar with the definition of the  ball box dimension (also known as Minkowski dimension)
will have no difficulty seeing that this definitions coincide
with the one given  in the introduction of this paper.

Now the main  result  by E.~M.~Bron\u ste\u\i n and L.~D.~Ivanov~\cite{bi}
asserts that for any convex set $\Omega$ inscribed in the unit Euclidean ball,
there are no more than $c(n)\varepsilon^{(1-n)/2}$ points whose convex hull is no
more than $\varepsilon$ away from $\Omega$ in the Hausdorff topology. Which gives

\begin{theo}[E.~M.~Bron\u ste\u\i n and L.~D.~Ivanov~\cite{bi}]\label{bitheo}
  Let $\Omega$ be a convex body in $\R^n$, then
$$
\overline{a}(\Omega)\leq \dfrac{n-1}{2}
$$
\end{theo}





\begin{thebibliography}{BBC07}
 \bibitem[Ale39]{alexandrov}
\bgroup\bf A.~D. Alexandroff\egroup{}.
 \newblock Almost everywhere existence of the second differential of a convex
   function and some properties of convex surfaces connected with it.
 \newblock {\em Leningrad State Univ. Annals [Uchenye Zapiski] Math. Ser.},
   6:3--35, 1939.

\bibitem[AB09]{alvarez_berck} \bgroup\bf J.C~\'Alvarez~Paiva\egroup{}  \andname{} \bgroup\bf G.~Berck\egroup{}.
\newblock What is wrong with the Hausdorff measure in Finsler spaces.  
\newblock\emph{Adv. Math.}  204(2):647--663, 2006.

\bibitem[AF98]{alvarez_fernandez} \bgroup\bf J.C~\'Alvarez~Paiva\egroup{} \andname{} \bgroup\bf E.~Fernandes\egroup{}.
\newblock Crofton formulas in projective Finsler spaces.  
\newblock \emph{Electron. Res. Announc. Amer. Math. Soc.}  4:91--100, 1998,







\bibitem[Ben03]{be}
\bgroup\bf Y.~Benoist\egroup{}. \newblock Convexes hyperboliques et fonctions quasi sym\'e\-tri\-ques.
\newblock {\em Publ. Math. Inst. Hautes \'Etudes Sci.}, 97:181--237, 2003

\bibitem[Ben06]{yvessurvey}
\bgroup\bf Y.~Benoist\egroup{}.
\newblock A survey on divisible convex sets.
\newblock Written for the Morningside center conference in Beijing 2006.



\bibitem[Ben60]{benzecri}
\bgroup\bf J.-P. Benz{\'e}cri\egroup{} 
\newblock Sur les vari\'et\'es localement
  affines et localement projectives
\newblock \emph{Bull. Soc. Math. France}
  88:p.~229--332, 1960.

\bibitem[BBV10]{berck_Bernig_Vernicos} \bgroup\bf G.~Berck\egroup{}, \bgroup\bf
  A.~Bernig\egroup{}  \andname{} \bgroup\bf
  C.~Vernicos\egroup{}.  \newblock Volume Entropy of Hilbert Geometries. 
\newblock \textsl{Pacific J. of Math.}, 245(2):201--225, 2010.


\bibitem[Ber77]{berger}
\bgroup\bf M.~Berger\egroup{}.
\newblock {\em G{\'e}om{\'e}trie}, volume 1/actions de groupes, espaces affines
  et projectifs.
\newblock Cedic/Fernand Nathan, 1977.

\bibitem[Ber09]{berck} \bgroup\bf G.~Berck\egroup{}.
\newblock Minimality of totally geodesic submanifolds in Finsler geometry.
\newblock  \emph{Math. Ann.}  343(4):955--973, 2009.

\bibitem[BO08]{Borisenko-Olin-2008}
\bgroup\bf A.~A. Borisenko\egroup{} \andname{} \bgroup\bf E.~A. Olin\egroup{}.
\newblock {Asymptotic Properties of Hilbert Geometry.}
\newblock \emph{Zh. Mat. Fiz. Anal. Geom. 5} 4(3):327--345, 443, 2008.

\bibitem[BO11]{Borisenko-Olin-2011}
\bgroup\bf A.~A. Borisenko\egroup{} \andname{} \bgroup\bf E.~A. Olin\egroup{}.
\newblock{Curvatures of spheres in Hilbert geometry.}
\newblock \emph{Pacific J. Math.} 254(2):257--273, 2011.

\bibitem[BI75]{bi} \bgroup\bf E.~M.~Bronshteyn\egroup{} \andname{} \bgroup\bf L.~D.~Ivanov\egroup{}.
  \newblock The approximation of convex sets by polyhedra
  \newblock \textsl{Siberian Math. J.}, 16(5):852--853, 1975.

\bibitem[BBI01]{bbi} \bgroup\bf D.~Burago\egroup{}, \bgroup\bf
  Y.~Burago\egroup{} \andname{} \bgroup\bf S.~Ivanov\egroup{}.
  \newblock \textsl{A Course in Metric Geometry}, \volumename{}~33
  \ofname{} \textsl{Graduate Studies in Mathematics}.  \newblock
  American Mathematical Society, 2001.



\bibitem[Bus55]{busemann}
\bgroup\bf H.~Busemann\egroup{}.
\newblock \emph{The geometry of geodesics}, Academic Press Inc., New York, N. Y., 1955.

\bibitem[BK53]{busemannkelly}
\bgroup\bf H.~Busemann\egroup{} \andname{} \bgroup\bf P.~J. Kelly\egroup{} --
  \emph{Projective geometry and projective metrics}, Academic Press Inc., New
  York, N. Y., 1953.

\bibitem[CV07]{ccv} \bgroup\bf B.~Colbois\egroup{} \andname{} \bgroup\bf
  C.~Vernicos\egroup{}.  \newblock Les g\'eom\'etries de Hilbert sont \`a g\'eom\'etrie locale born\'ee
\newblock \textsl{Annales de l'institut Fourier}, 57(4):1359--1375, 2007.

\bibitem[CV04]{cpv} \bgroup\bf B.~Colbois\egroup{} \andname{} \bgroup\bf
  P.~Verovic\egroup{}.  \newblock Hilbert geometry for strictly
  convex domain.  \newblock \textsl{Geom. Dedicata}, 105:29--42, 2004

\bibitem[CVV04]{cvv1} \bgroup\bf B.~Colbois\egroup{}, \bgroup\bf
  C.~Vernicos\egroup{}  \andname{} \bgroup\bf
  P.~Verovic\egroup{}.  \newblock L'aire des triangles id\'eaux en g\'eom\'etrie de Hilbert.
\newblock \textsl{Enseign. Math.}, 50(3--4):203--237, 2004.

\bibitem[CVV06]{cvv2} \bgroup\bf B.~Colbois\egroup{},  \bgroup\bf
  C.~Vernicos\egroup{}  \andname{} \bgroup\bf
  P.~Verovic\egroup{}.  \newblock Area of ideal triangles and gromov hyperbolicity
in Hilbert geometries.
\newblock  \textsl{Illinois J. of Math.}, 52(1):319--343, 2008.

\bibitem[CLT]{Cooper-Long-Tillmann}
\bgroup\bf D.~Cooper\egroup{}, \bgroup\bf D.~Long\egroup{} \andname{} \bgroup\bf S.~Tillmann\egroup.
\newblock {On Convex Projective Manifolds and Cursps}
\newblock {arXiv:1109.0585v2} to appear in \emph{Advances in Math}. 

\bibitem[Cra09]{crampon} \bgroup\bf M.~Crampon\egroup{}. 
\newblock Entropies of compact strictly convex projective manifolds.
\newblock \emph{Journal of Modern Dynamics}, Volume 3, No. 4,  2009. 

\bibitem[CM14]{Crampon-Marquis-2014} 
\bgroup\bf M.~Crampon\egroup{} \andname{} \bgroup\bf L.~Marquis\egroup{}.
\newblock {Le flot g{\'e}od{\'e}sique des quotients g{\'e}om{\'e}triquement finis des
 g{\'e}om{\'e}tries de Hilbert.}
\newblock \emph{Pacific J. Math.}  268(2):313--369, 2014.

\bibitem[dlH93]{dlharpe} \bgroup\bf P.~de~la Harpe\egroup{}.
  \newblock On {H}ilbert's metric for simplices.  \newblock \Inname{}
  \textsl{Geometric group theory, Vol.\ 1 (Sussex, 1991)},
  \pagesname{} 97--119. Cambridge Univ. Press, 1993.

\bibitem[FT48]{fejes} \bgroup\bf L.~Fejes~T\'oth\egroup{}
  \newblock Approximation by polygons and polyhedra.
  \newblock \textsl{Bull. Amer. Math. Soc.}, 54:431--438, 1948.

\bibitem[FK05]{foertsch-karlsson05}
\bgroup\bf T.~Foertsch\egroup{} \andname{} \bgroup\bf A.~Karlsson\egroup{}.
\newblock {Hilbert metrics and Minkowski norms.}
\newblock {\em J. Geom.}, 83(1-2):22--31, 2005.


\bibitem[Gol90]{Goldman-1990}
\bgroup\bf W.~Goldman\egroup{}.
\newblock {Convex real projective structures on compact surfaces} 
\newblock \emph{J. Differential Geom.} 31:126--159, 1990. 

\bibitem[Gru07]{gruber} \bgroup\bf
 P.~M.~Gruber\egroup{}
\newblock \textsl{Convex and discrete geometry},
\newblock  Grundlehren der Mathematischen Wissenschaften [Fundamental Principles
 of Mathematical Sciences], 336. 
\newblock Springer, Berlin,  2007. xiv+578 pp. ISBN: 978-3-540-71132-2 


\bibitem[Hil71]{dhilbert} \bgroup\bf D.~Hilbert\egroup{}.  \newblock
  \textsl{Les fondements de la G\'eom\'etrie}, \'edition critique
    pr\'epar\'ee par P.~Rossier.  \newblock Dunod, 1971.

\bibitem[Joh48]{John} \bgroup\bf F.~John\egroup{}.  \newblock Extremum
  problems with inequalities as subsidiary conditions.  \newblock
  \Inname{} \textsl{Studies and Essays Presented to R. Courant on his
    60th Birthday, January 8, 1948}, \pagesname{} 187--204.
  Interscience Publishers, Inc., New York, 1948.

\bibitem[Kay67]{dckay}
\bgroup \bf D.~C. Kay\egroup{}.
\newblock The ptolemaic inequality in {H}ilbert geometries.
\newblock {\em Pacific J. Math.}, 21:293--301, 1967.

\bibitem[KN02]{Karlsson-Noskov-2002}
\bgroup\bf A.~Karlsson\egroup{} \andname{} \bgroup\bf G.A. Noskov\egroup{}.
\newblock {The Hilbert metric and Gromov hyperbolicity.}
\newblock {\em Enseign. Math.}, 48(1-2):73--98, 2002.

\bibitem[Kim05]{Kim-2005}
\bgroup\bf I.~Kim\egroup{}.
\newblock {Compactification of strictly convex real projective structures}.
\newblock \emph{ Geom. Dedicata } 113:185--195, 2005. 

\bibitem[LW11]{Lemmens-Walsh-2011}
\bgroup \bf B.~Lemmens\egroup{} \andname{}  
\bgroup\bf C.~Walsh\egroup{}. 
\newblock {Isometries of polyhedral Hilbert geometries}.
\newblock \emph{J. Topol. Anal.}  3(2):213--241, 2011.		

  \bibitem[Lev97]{levy} \bgroup\bf S.~Levy\egroup{}.  \newblock
  \textsl{Flavors of Geometry}, \volumename{}~31 \ofname{}
  \textsl{Mathematical Sciences Research Institute Publications}.
  \newblock Cambridge Univ. Press, 1997.

\bibitem[LN08]{Lins-Nussbaum-2008}
\bgroup\bf B.~Lins\egroup{} \andname{} \bgroup\bf R.~Nussbaum\egroup. 
\newblock {Denjoy-Wolff theorems, Hilbert metric nonexpansive maps and
 reproduction-decimation operators.}
\newblock \emph{J. Funct. Anal.}  254(9):2365--2386, 2008.

\bibitem[Mar12]{Marquis-2012}
\bgroup\bf L.~Marquis\egroup{}.
\newblock{Surface projective convexe de volume fini.}
\newblock \emph{Ann. Inst. Fourier (Grenoble)}  62(1):325--392, 2012.

\bibitem[McM71]{mcmullen}  \bgroup\bf P.~McMullen\egroup{}.  
\newblock On the upper bound conjecture for convex polytopes.
\newblock Journal of Combinatorial Theory, 10(3):187--200, 1971.

\bibitem[MS71]{mcmullens} \bgroup\bf P.~McMullen and G.C.~Shephard\egroup{}.  
\newblock \textsl{Convex polytopes and the upper bound conjecture}, \volumename{}~3 
\ofname{} \textsl{London mathematical society lecture notes series}.
 \newblock Cambridge Univ. Press, 1971.

\bibitem[Nas61]{nasu61}
\bgroup\bf Y.~Nasu\egroup{}.
\newblock {On Hilbert geometry.}
\newblock {\em Math. J. Okayama Univ.}, 10:101--112, 1961.

\bibitem[Nie]{Nie-2011}
\bgroup\bf X.~Nie\egroup{}
\newblock {On the Hilbert geometry of simplicial Tits sets.}
\newblock {arXiv:1111.1288v4}, july 2014.

\bibitem[Hbk14]{Handbook}
\newblock \emph{Handbook of Hilbert geometry}
\newblock editors A.~Papadopoulos and M.~Troyanov, 
\newblock {European Mathematical Society Publishing House Z\"urich}, 2014.


\bibitem[SW81]{sw} \bgroup\bf R.~Schneider\egroup{} \andname{}
 \bgroup\bf J.~A.~Wieacker\egroup{}.  \newblock Approximation of convex bodies by Polytopes.
\newblock \textsl{Bull. London Math. Soc.}, 13:149--156, 1981.


\bibitem[SM02]{so2} \bgroup\bf E.~Soci\'e-Methou\egroup{}. Caract\'erisation des ellipso\"\i des par
  leurs groupes d'automorphismes. \textsl{Ann. Sci. de l'{\'E}NS},
  35(4):537--548, 2002.

\bibitem[Tom]{essay17} \bgroup\bf A.~Thompson\egroup{}.
essay 17: Lengths of curves. \textsl{in a book in progress}, 2012.

\bibitem[Ver09]{ver09} \bgroup\bf C.~Vernicos\egroup{}.
\newblock Spectral Radius and amenability in Hilbert Geometry.
 \newblock \textsl{Houston Journal of Math.}, 35(4):1143-1169, 2009.

\bibitem[Ver11]{ver2011} \bgroup\bf C.~Vernicos\egroup{}.
\newblock On the Hilbert geometry of product.
 \newblock arXiv:1109.0187v1, to appear in \textsl{Proceedings of the AMS}.

\bibitem[Ver13]{ver2012} \bgroup\bf C.~Vernicos\egroup{}.
\newblock The asymptotic volume in Hilbert geometries.
 \newblock \textsl{Indiana U. J. of math.}, 62(5):1431-1441, 2013.

\end{thebibliography}
\end{document}